\documentclass[a4wide,11pt]{amsart}

\newcommand{\old}[1]{}

\usepackage{amsmath}
\usepackage{amssymb}
\usepackage{amsfonts}
\usepackage{subfig}
\usepackage{graphicx}
\usepackage{psfrag}

\usepackage{epstopdf}
\usepackage{a4wide}
\usepackage{color}

\usepackage{url}

\usepackage{epic}
\usepackage{latexsym}
\usepackage{amsmath}
\usepackage{epsfig}
\usepackage{amssymb}
\usepackage{enumerate}

\usepackage[mathscr]{euscript}

\usepackage[all]{xy}

\newtheorem{lemma}{Lemma}[section]
\newtheorem{theorem}[lemma]{Theorem}
\newtheorem{proposition}[lemma]{Proposition}
\newtheorem{corollary}[lemma]{Corollary}

\newcommand{\cL}{{\mathcal L}}

\newcommand{\cR}{{\mathcal R}}

\newtheorem{observation}{Observation}

\newcommand{\steven}[1]{\textcolor{black}{#1}}
\newcommand{\kelk}[1]{\textcolor{black}{#1}}

\begin{document}
\title[Reconstructing level-1 phylogenetic networks from triplets and clusters]{On the challenge of reconstructing level-1 phylogenetic networks from triplets and clusters}
\author{P. Gambette, K.T. Huber, S. Kelk}
\thanks{School of Computing Sciences, 
University of East Anglia, UK, \\
LIGM, Universit{\'e} Marne-la-Vall{\'e}e, France,\\
Department of Knowledge Engineering (DKE), Maastricht University, 
The Netherlands}

\date{4th April, 2016}
\begin{abstract}
Phylogenetic networks have gained prominence over the years due to 
their ability to represent complex non-treelike evolutionary events
such as recombination or hybridization. Popular combinatorial objects
used to construct them are triplet systems 
and cluster systems, the motivation being 
that any network $N$ induces a triplet system $\mathcal R(N)$ and a
softwired cluster system $\mathcal S(N)$. Since in real-world studies
it cannot be guaranteed that all triplets/softwired 
clusters induced by a network
are available, it is of particular interest to understand whether subsets of
  $\mathcal R(N)$ or  $\mathcal S(N)$
 allow one to uniquely reconstruct the underlying network $N$.
Here we show that even within the highly restricted yet
biologically interesting space of level-1 
phylogenetic networks   it is not always possible to uniquely 
reconstruct a level-1 network $N$\kelk{,} even when all triplets in  
$\mathcal R(N)$ or all clusters in $\mathcal S(N)$ 
are available. On the positive side, we introduce a
 reasonably large subclass of level-1
networks the members of which are uniquely determined by their
induced triplet/softwired cluster systems.
Along the way, we also establish various enumerative results, both 
positive and negative, including results which show
that certain special subclasses of level-1 networks $N$ can be uniquely 
reconstructed from proper subsets of
$\mathcal R(N)$ and $\mathcal S(N)$. We anticipate these results to
be of use in the design of 
algorithms for phylogenetic network inference.
\end{abstract}

\maketitle

\section{Introduction}
\kelk{Phylogenetic trees are essentially graph-theoretical
trees whose set of leaves is labelled by a set of species or organisms (more abstractly, \emph{taxa}) and which do not have any
degree-two vertices, except possibly the root. They have been the 
model of choice for many years for shedding light \steven{on} the 
evolutionary past of a set of taxa.} \kelk{However, in} cases where the taxa are suspected to 
have undergone reticulate evolutionary events such as
hybridization or recombination\kelk{, trees} have been
found to not always be appropriate \cite{S75}. The need for structures
capable of appropriately dealing with such data sets\kelk{,} combined with the
fact that different evolutionary processes have given rise to them\kelk{,}
has resulted in the introduction of a number of more general 
structures for representing evolutionary relationships. \kelk{Subsumed}  
under the  name \kelk{``phylogenetic network''} these include 
\kelk{hybrid phylogenies} \cite{BSS06}, \kelk{ancestral} recombination graphs \cite{H90},
galled trees \cite{gusfield2004optimal,WZZ01}, 
normal networks \cite{W10},
regular networks \cite{BSS04}, tree-sibling
networks \cite{CLRV08}, level-$k$ networks
\cite{JNS06,IKKSHB09},
median networks \cite{B94}
and NeighborNets \cite{BM03}, to name just a few, which
all generalize a phylogenetic tree in one way or another.

Apart from median networks and NeighborNets which 
are a special type of split-based phylogenetic network, the basic
graph-theoretical structure underpinning a phylogenetic network
is a rooted directed acyclic graph (DAG) that has a unique root 
and whose set of sinks is a given set of taxa.
\kelk{One} of the combinatorially simplest types of 
phylogenetic network\kelk{,} but still
complicated enough to be of interest to Evolutionary Biology\kelk{,} is that
of a binary level-$1$ network (see Fig.~\ref{fig:blobby-shapely} 
for an example).
\begin{figure}[!ht]
\includegraphics[scale=0.4] {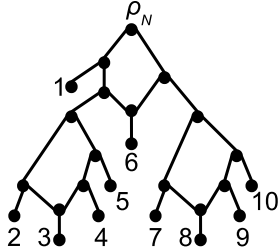} 
\caption{\label{fig:saturated-shapely} A \steven{binary level-1} phylogenetic network $N$
on  $X=\{1,\ldots, 10\}$ \steven{that is also 4-outwards and saturated. As in all figures all
arcs of \kelk{the} network are directed downwards, so we do not explicitly indicate the direction of arcs.}}
\label{fig:blobby-shapely}
\end{figure}
\textcolor{black}{Such structures} have attracted a considerable amount of interest in
the literature (see e.\,g.\,\cite{JNS06,gusfield2004optimal,rossello2009all,HIKS2011})
and can informally be thought of as \kelk{degree-constrained}
rooted DAGs with \steven{vertex-disjoint undirected cycles}. \kelk{(Formal definitions of all terms will follow in later sections)}.
However, this simplicity
has proven to be deceptive\steven{,} as the combinatorial structure 
of such networks has turned out to be more complicated than originally
thought (see e.g. \cite{GH12,HM13}). \kelk{Limits on our ability to reconstruct
level-1 networks constitute lower bounds on how well we can reconstruct
phylogenetic networks more generally. On the other hand, positive results for reconstructing level-1 networks can be an important first step towards algorithms for reconstructing more complex phylogenetic networks.}



In this paper, we \kelk{start by establishing a number of enumerative results for binary level-1 networks. These include upper and lower bounds on the number of vertices and arcs in such networks. We gradually shift our focus onto} 
{\em cluster systems}, that is, collections of non-empty
subsets of \kelk{the leaves}, \kelk{and} {\em triplet systems}, that is,
binary phylogenetic trees on just three leaves. \kelk{Guided by the fact that
these systems have been used} for 
reconstructing phylogenetic networks (see e.\,g.\,\cite{M09} and \cite{HRS10} 
for \steven{recent overviews}), 
we are particularly interested in finding bounds on the minimum size of a triplet
system/cluster system \kelk{required} to ``uniquely determine'' a level-1 network.
For trees this question is well understood. Specifically, for a phylogenetic tree $T$ on $n \geq 3$ leaves 
it is well-known that
$T$ is uniquely determined 
by its induced triplet system $\mathcal R(T)$ (leading to an
upper bound of ${n\choose 3}$ for such a minimum-sized set)
and that $n-2$ carefully chosen triplets from $\mathcal R(T)$
suffice to uniquely reconstruct $T$ when $T$ is binary (see Theorem 3 of \cite{Steel1992} 
and its Corollary). For  this case, it is also well-known that
 $T$ is uniquely determined 
  by its induced cluster system $\mathcal C(T)$
and that 
for a minimum-sized cluster system to uniquely determine $T$, 
it must have $|\mathcal C(N)|=2n-1$ elements.

As we shall see,
the situation is more complicated for binary level-1 networks.
Every level-1 network $N$
induces a triplet system $\mathcal R(N)$ and a certain
cluster system $\mathcal S(N)$ called the {\em softwired cluster system
of $N$} (see \cite{huson2011survey} for background) but
their ability to fully capture the topological structure of $N$ is
not as strong as one might hope.
Let us say that a binary level-1 network $N$ is {\em encoded}
by its induced triplet system if for every binary level-1 
network $N'$ such that $\mathcal R(N')=\mathcal R(N)$, we have $N = N'$. Continuing,
we say that a binary level-1 network is {\em 4-outwards} if its underlying graph  
does not have a cycle of length four or less. It is precisely 
the 4-outwards binary level-1 networks $N$ that 
are encoded
by $\mathcal R(N)$ as well as $\mathcal S(N)$ \cite{GH12} 
(where we define a binary level-1 network
to be encoded by its induced softwired  cluster system in an
analogous way). 

Intriguingly, if $\mathcal R(N')=\mathcal R(N)$ 
is replaced by $\mathcal R(N)\subseteq \mathcal R(N')$ (as is the
case in our formalization of ``uniquely determining'') then the 
assumption that $N$ is $4$-outwards is no longer strong enough
to guarantee uniqueness.  A similar observation holds for $\mathcal S(N)$
(see Sections~\ref{sec:triplets-define} and 
\ref{sec:define-clusters} for examples for both cases).  
However, the situation changes for both if, in addition to
being $4$-outwards we require that $N$ is {\em saturated}, that is,
none of its vertices is incident with more than one cut arc  
(Theorem~\ref{theo:shapely-saturated-triplet} and 
Theorem~\ref{theo:shapely-saturated-cluster}). 
Simple networks on $n\geq 4$ leaves are 
4-outwards, saturated networks that have precisely one cycle
in their underlying graph. \steven{We show that at most $2n-1$ 
carefully chosen triplets suffice to uniquely determine such networks}.
As the network on four leaves depicted in Fig.~\ref{fig:new} indicates, 
this bound is however not tight because five triplets suffice in that case (which can be checked by a simple case analysis).  
Given that any binary level-1 network $N$ contains at least one triplet for any 
three of its leaves and so $|\mathcal R(N)|\geq {n\choose 3}$ holds, this suggests
that at least for simple phylogenetic networks there is a considerable amount
of redundancy in $\mathcal R(N)$ with regards to reconstructing $N$
from  $\mathcal R(N)$. To establish a similar result for
general binary level-1 networks $N$ might not be \steven{straightforward} 
in view of Proposition~\ref{prop:triplet-system-size}\steven{,} 
which suggests that  $|\mathcal R(N)|$ is not easily expressible
in terms of a natural parameter associated \steven{with} a phylogenetic network $N$\steven{,} namely
its number of non-trivial cut arcs (see Section~\ref{arcs-vertices-galls}).
 This is somewhat surprising in view of the close
relationship between the triplet system induced by a binary  level-1 network
$N$ and its associated softwired cluster system $\mathcal S(N)$ 
(see e.\,g.\,\cite[Proposition 2 and Theorem 1]{GH12} for details 
concerning this relationship) because
the size of $\mathcal S(N)$ \emph{is} closely related to the number
of cut arcs of $N$ (Theorem~\ref{theo:countingClusters}). As in the
case of triplet systems,
it is easy to find examples of binary level-1 networks $N$ that indicate
that there is redundancy in the 
softwired cluster system induced by $N$ with regards to uniquely determining 
$N$. Again focusing on simple networks $N$, we show 
that at most $n$ carefully chosen
(softwired) clusters induced by $N$ suffice to uniquely determine $N$
(Corollary~\ref{cor:defineSimple}). However, we do not know if this
bound is sharp.

Given that in phylogenetic \steven{analyses}
one is hardly ever guaranteed to have all triplets/clusters induced by a 
(as yet unknown) phylogenetic network \steven{available}, the above observations 
have profound consequences for phylogenetic network reconstruction. 
One of the most important \steven{ones is} that 
a phylogenetic network reconstructed \steven{from} a triplet or cluster system
need not be the network that gave
rise to \steven{this system}.

The paper is organized as follows. In the next section, we present
basic terminology of relevance to this paper\steven{,} including the definition of
a level-$k$ network and that of a gall in a level-1 network. In 
Section~\ref{arcs-vertices-galls}, we define cut arcs and
present formulas for counting the
number of vertices, arcs, and galls in a binary level-1 network. 
These results
improve on the results in \cite{CJSS05} which imply that the number
of vertices in a binary level-1 network on $n$ leaves is linear in $n$
and that the number of hybrid vertices is at most $n-1$. In
Section~\ref{counting-softwired}, we formally define the
softwired cluster system $\mathcal S(N)$ 
induced by a binary level-1 network $N$
and establish Theorem~\ref{theo:countingClusters}.
In Section~\ref{counting-triplets}, we define the triplet 
system $\mathcal R(N)$
induced by a binary level-1 network $N$ and establish 
Proposition~\ref{prop:triplet-system-size}. In Section~\ref{sec:triplets-sn},
we establish in Proposition~\ref{prop:samecut}
a relationship between the triplet system
induced by a binary level-1 network $N$ and a certain partition 
of the leaf set of $N$ 
that will be crucial for showing
Theorem~\ref{theo:shapely-saturated-triplet}.   
In Section~\ref{sec:triplets-define}, we first formalize the 
notion of ``uniquely 
determining'' and then present the aforementioned examples for triplet 
systems. Starting
in that section and continuing in Section~\ref{sec:define-clusters}, 
we investigate saturated, 
4-outwards, binary level-1 networks
and establish Theorem~\ref{theo:shapely-saturated-triplet} and 
Theorem~\ref{theo:shapely-saturated-cluster}, respectively.


\section{Definitions and Notation}\label{notation}

\noindent
\kelk{In this section we present only basic definitions and notation to avoid overloading the
reader. Concepts such as triplets and (softwired) clusters are formalized in subsequent sections.}

Throughout the paper, let $X$ denote a finite set of size  
$n\geq 2$. Also all graphs $G$ considered have non-empty
finite sets of vertices and edges (or arcs in case $G$ is
directed) and have
no loops or multiple edges (or arcs in case $G$ is
directed). 

Suppose for the following that 
$G=(V,A)$ is a directed acyclic graph (DAG). 
If $v$ and $w$ are vertices of $G$ such that
there exists an arc $a$ from $v$ to $w$ in $G$ then we denote that 
arc by $(v,w)$ and refer to $v$ as the {\em tail of $a$}, denoted by
$tail(a)$, and to $w$ as the {\em head of $a$}, denoted by $head(a)$.
Suppose $v\in V$ is a vertex of $G$.
Then we denote by
\kelk{$outdeg(v)$} the out-degree of $v$ \steven{(i.e. the number of arcs whose tail is
incident to $v$)} and by \kelk{$indeg(v)$} the
{\em in-degree} of $v$ \steven{(the number of arcs whose head is incident to $v$)}. The sum of the
out-degree and the in-degree of $v$ is called the {\em degree} of 
$v$\steven{,} denoted by \kelk{$deg(v)$}.  If $indeg(v)=1$ and $outdeg(v)=0$ then
$v$ is called a {\em leaf} of $G$. The set of leaves of $G$ is denoted by
$L(G)$. Every vertex in $V-L(G)$ is called an {\em interior vertex}
of $G$. 
 If $G$ has a unique vertex $\rho=\rho_G\in V$
with $indeg(\rho)=0$ and $outdeg(\rho)\geq 2$ then $\rho$ is called
the {\em root} of $G$ and $G$ is 
called a {\em rooted DAG}. If $G$ is a rooted DAG with leaf set
$X$ and  $G'=(V',A')$ is a further rooted DAG  
with leaf set $X$ then we say that $G$ is {\em equivalent} to $G'$ 
\steven{if there exists a graph isomorphism from $G$ to $G'$ that is the
identity on $X$}.


A {\em phylogenetic network $N$ on $X$} 
is a rooted DAG whose set of leaves is $X$,
and every interior vertex $v$ of $N$ except the root $\rho_N$ is either (i) a
{\em \steven{split} vertex} of $N$, that is, $indeg(v)=1$ and $outdeg(v)\geq 2$
or (ii) a {\em hybrid vertex} of $N$, that is,  $indeg(v)\geq 2$ and 
$outdeg(v)\geq 1$. In case only the size of $X$ 
is of relevance to the discussion then we will simply call
$N$ a {\em phylogenetic network on $|X|$ leaves} and if the set $X$ is of 
no relevance to the discussion then we will simply call a phylogenetic 
network $N$ on $X$ a {\em phylogenetic network}.  
We denote the set of hybrid vertices of 
a phylogenetic network $N$ by $H(N)$
and say that $N$ is {\em binary}
if the root of $N$ as well as every \steven{split} vertex 
of $N$ has out-degree two and
\steven{every hybrid vertex of $N$ has out-degree one and in-degree 2}. 


\steven{An undirected} graph $G$ is  called 
{\em biconnected} if $G$ is connected and $G - v$ is connected
for all $v \in V(G)$.
\steven{A maximal biconnected subgraph $H$ of $G$ is called
a \emph{biconnected component} of $G$. (We say that a biconnected component is
\emph{non-trivial} if it contains more than one edge).} \kelk{Let $U(N)$ be the \emph{underlying graph} of $N$ i.e. the undirected graph obtained
from $N$ by ignoring the orientation of its arcs.} \kelk{We say that a binary phylogenetic network $N$ is}
 a {\em level-$k$ (phylogenetic) network},
if every biconnected component of $U(N)$ contains
at most $k$ hybrid vertices. 
Reflecting the fact 
that a cycle of length three in the underlying graph of a 
phylogenetic network is indistinguishable 
(from a triplet or cluster perspective)
from a \steven{split} vertex, we follow common practice and 
will always assume that a cycle in the
underlying graph of a level-1 network $N$ contains at 
least four vertices. 

 Note that
a phylogenetic network $N$
for which $H(N)=\emptyset$ holds is simply a {\em rooted phylogenetic tree} on $X$ (sensu
\cite{SS03}).
Thus,  
level-0 networks are rooted phylogenetic trees. All phylogenetic trees considered
in this article are rooted so we henceforth drop the ``rooted'' prefix.

We denote the class of all binary
level-1 networks on $n\geq 2$ leaves by $\cL_1(n)$. Alternatively,
we will also use $\cL_1(X)$ to denote that class if we
want to emphasize the leaf set $X$ of the networks in  $\cL_1(n)$.

Now, suppose that $N$
is a level-$k$ network, $k\geq 1$.  Then we call $N$ 
{\em proper} if $N$ is not also a level-$l$ network for some
$0\leq l\leq k-1$. Note that in case $k=1$ such a network
must have at least three leaves and at least one hybridization vertex. In that case, we call a \steven{non-trivial} biconnected component 
of $U(N)$ with its original directions in $N$ restored a {\em gall}
of $N$ and denote the set of galls of a level-1 network $N$
by $\mathcal G(N)$. If
$N$ is binary, contains precisely one gall $C$, and every
leaf of $N$ is adjacent with a vertex of $C$ then 
$N$ is called {\em simple}. Together with phylogenetic trees,
such networks may be viewed as the building blocks of
(proper) level-1 networks \cite{IKKSHB09}. For the convenience of
the reader, we present examples of
two simple level-1 networks on $X=\{x_1,\ldots, x_5\}$ in 
Fig.~\ref{fig:simple}.
%
%
\begin{figure}[!ht]
\begin{tabular}{ccc}
\includegraphics[scale=0.4] {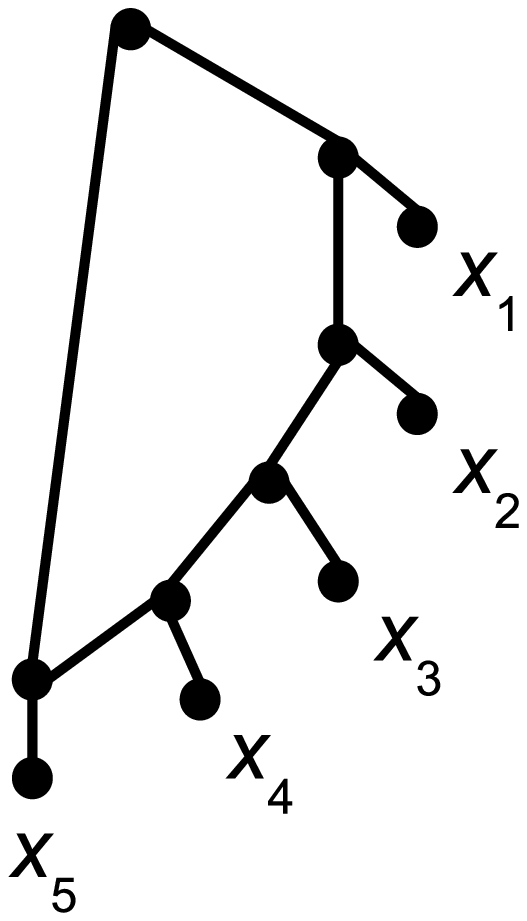} & 
\includegraphics[scale=0.4] {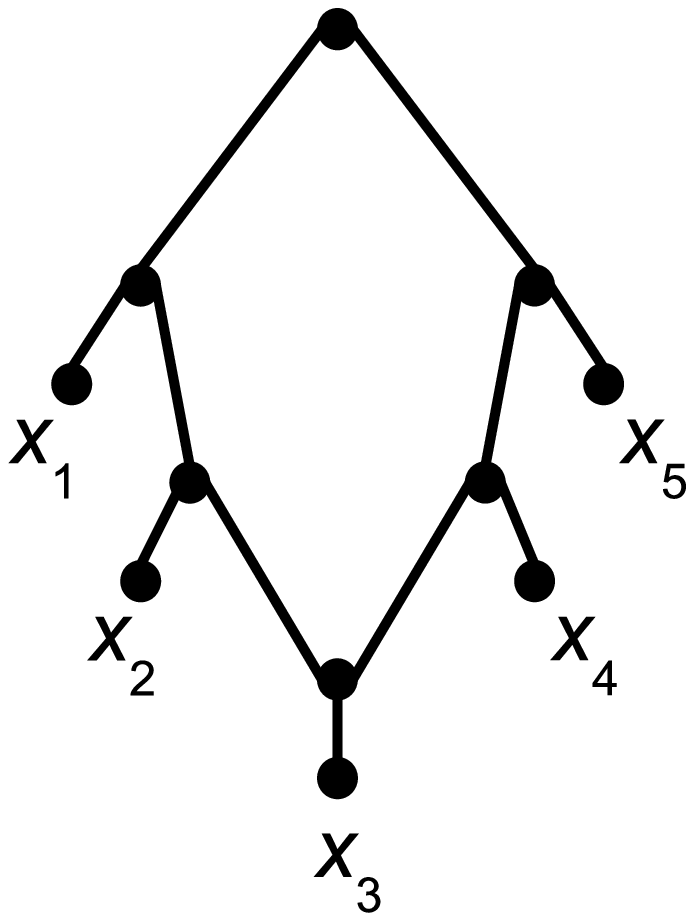} \\
(i) & (ii)
\end{tabular}
\caption{\label{fig:simple} Two examples of simple level-1 networks
on $X=\{x_1,\ldots, x_5\}$. Note that both networks are
4-outwards and saturated.
}
\end{figure}
 
\section{Counting arcs, vertices and galls}\label{arcs-vertices-galls}
In this section, we present some enumerative results concerning the
number of vertices, arcs, and galls of a  level-1 network.
We start by introducing some relevant notation. Suppose $N$
is a phylogenetic network on $X$. Following \cite{ISS10}, 
we say that a phylogenetic
tree $T$ on $X$ is  {\em displayed} by $N$ if there exists a 
subgraph $N'$ of $N$ that is a \emph{subdivision} of $T$ \steven{i.e. $T$ can be obtained
from $N'$ by repeatedly suppressing vertices with in-degree and out-degree both equal to 1.}
For $N$
a level-1 network, we denote the number of
galls of $N$ by $g(N)$, that is, we \steven{let} $g(N)=|\mathcal G(N)|$.

\subsection{Counting arcs and vertices}

In case $N$ is a binary level-0 network on $n\geq 2$ 
leaves, that is, $N$ is a 
\textcolor{black}{binary} phylogenetic tree on $n$ leaves, it is easy to see 
 that $N$ has $2n-1$ 
vertices and $2(n-1)$ edges (see e.\,g.\,\cite[Proposition 2.1.3]{SS03}
for the corresponding result for unrooted binary phylogenetic trees). 
For the more general case that 
$N$ is a binary, proper, level-$k$ network on $n\geq 2$ 
(and thus on $n\geq 3$) leaves, \steven{and} $k\geq 1$, 
 it was shown in \cite[Lemma 4.5]{I09}
that any such network can contain at most $2n-1+k(n-1)$
vertices and at most $2n-2+\frac{3}{2}k(n-1)$ arcs. 
Denoting for $n\geq 3$
the \steven{sub}class of all proper level-1 networks in  
 $\cL_1(n)$ by $\cL_1(n)^-$, 
the sizes of the vertex and arc \steven{sets} of a network $N=(V,A)$ in 
$\in\cL_1(n)^-$ can \steven{thus} be at most $3n-2$ 
and $3.5(n-1)$, respectively. Moreover,
if follows from  \cite[Lemma 4.4]{I09} that $|V|=2n+1=|A|$
holds in the special case that $N$ is simple. The next
result indicates that the size of the vertex set of a simple level-1 network
lends itself to providing lower bounds on the sizes of the vertex set and
arc set of a general  proper level-1 network, respectively.
 
\begin{lemma}
\label{lem:countbounds}
Let $n\geq 3$ and suppose $N = (V,A)\in\cL_1(n)^-$. 
Then $2n + 1 \leq |V| \leq 3n-2$
and $2n + 1 \leq |A| \leq 3.5(n-1)$. \steven{These} bounds 
are tight if $n=3$\steven{,} 
in which case $N$ must be a simple
level-1 network.
\end{lemma}
\begin{proof}
\steven{Suppose $X$ has size $n$ and assume that 
$N=(V,A)$ is a network in $\cL_1(n)^-$ with leaf set $X$. 
The upper bounds have already been established in the above discussion, so it suffices to prove that 
$2n+1\leq |V|$ and $2n+1\leq |A|$. Let $g=g(N)$
and note that $g\geq 1$ holds because $N$ is assumed to be proper.
We start by adding a new taxon $y \not \in X$ just above the root of $N$, in the following way: introduce
a new vertex $u$, add an arc from $u$ to the root of $N$, and add an arc from $u$ to $y$. Let $N^{y}$ be this
new network, whose root is $u$. $N^{y}$ has $|V|+2$ vertices and $|A|+2$ arcs.
Let $T = (V_T, A_T)$ be any binary phylogenetic tree on $X \cup \{y\}$ that is
displayed by $N^{y}$. 
Then there exists a subgraph $T'$ of $N^{y}$ that
is a subdivision of $T$. Observe that $T'$ must contain
$2g$ vertices whose in-degree and out-degree (in $T'$) are both equal to 1. Specifically,
$g$ of them are hybrid vertices of $N^{y}$ and the other
$g$ are tails of arcs in $N^{y}$ whose head is contained in $H(N^{y})$. 
(The correctness of these claims requires the root of $T'$ to be the same vertex as the root of $N^{y}$, and this is the reason for the addition of $y$ in the first place.)
Consequently, $|V_T|=(|V|+2)-2g$. Noting that $T$ has $2(n+1)-1$ vertices (i.e. because it is binary) we have,
$$
|V| =|V_T|+2g-2
=(2(n+1)-1)+2g - 2
=2n - 1 + 2g
\geq 2n+1,
$$
where the last inequality follows because $g \geq 1$.}
\steven{
Similarly, to obtain $T'$ from $N^{y}$, we need to delete
for every hybrid vertex $h\in H(N^{y})$ precisely 
one of its incoming arcs $(v,h)$.  Hence, both $v$ and $h$ will 
have in-degree and out-degree 1 in $T'$. (Note
that $v$ might be the root of the gall that contains $h$ in $N^{y}$).
 Hence, $|A_T| = (|A|+2)-g-2g$. Noting (again, because it is binary)
that $T$ has $2(n+1)-2$ arcs, we have
$$
|A|=|A_T|+3g - 2
= 2(n+1)-2 + 3g - 2
= 2n + 3g - 2
\geq 2n + 1
$$
where, as before, the last inequality follows because $g \geq 1$.}

\steven{It can easily be verified that the bounds are tight in the case $n=3$. Specifically, all expressions evaluate
to 7. Finally, if $n=3$ then $N$ must be a simple level-1 network, because otherwise it would either be a tree (and thus not proper) or violate the assumption that every cycle in the 
underlying graph of $N$ contains at 
least four vertices.} 
\end{proof}

\subsection{Counting galls}

We next establish a formula for counting the number 
of galls of a  level-1 network. To this end, we require further
terminology. Suppose $G=(V,A)$ is a rooted DAG. Then an
arc $a \in A$ is called a {\em cut arc} of $G$ if the deletion of $a$ 
disconnects the underlying graph
$U(G)$ of $G$. If $a$ is a cut arc of $G$ such that $head(a)$ is 
a leaf of $G$ then we call $a$ a {\em trivial} cut arc of $G$
and a {\em non-trivial} cut arc of $G$ otherwise.
We denote the number of non-trivial cut arcs of a level-1 network $N$ 
by $c_N$. 

Suppose $N$ is a level-1 network on $X$. For a gall $C$ of $N$, we call
an arc of $N$ whose tail but not its head is a vertex of $C$ an
{\em outgoing arc} of $C$. Note that our assumption \steven{that every cycle in $U(N)$ has
at least four vertices} 
implies that every gall must have at least
three outgoing arcs. Moreover, if $N$ is binary
then we call two distinct leaves $x$ and $y$ of $N$ a {\em cherry} of 
 $N$ if $x$ and $y$ have a common parent. Note that that 
 parent must be a \steven{split} vertex of $N$. In addition, if $N$ is a binary
phylogenetic tree 
 and $|X|=3$ then $N$ is called a {\em triplet (on $X$)}. Saying that
a vertex $v$ of a rooted DAG $G$ is {\em below} a vertex $w$ of $G$ if
$w$ lies on a directed path from the root of $G$ to $v$ but is distinct
from $v$, we denote a triplet $t$ on $X=\{a,b,c\}$ for
which  the last common ancestor of
$a$ and $b$ is below the root of $t$ by
$ab|c$ (or equivalently $c|ab$). Finally, a collection
$\mathcal R$ of triplets is called a {\em triplet system (on
$\bigcup_{t\in\mathcal R} L(t)$)}.

\begin{theorem}\label{theo:gall-number}
Let $n\geq 2$ and suppose $N\in\cL_1(n)$. 
Then  $g(N)\leq n - c_N - 2$ and 
this bound is tight if either $N$ is a phylogenetic
tree or every gall of $N$ has exactly three outgoing arcs.
\end{theorem}
\begin{proof}
We prove the theorem by induction on $n\geq 2$.
Suppose $N\in \cL_1(n)$.
Then the stated inequality clearly holds in the form of an
equality for $n=2$ since in
that case $N$ is a phylogenetic tree. It also holds for $n=3$ because in 
that case $N$ is either a triplet and so has one non-trivial
cut arc but no gall, or $N$ is a simple level-1 network and so has precisely
one gall but no non-trivial cut arcs.

Suppose that $N$ has $n\geq 4$ leaves 
and assume that $g(N')\leq n-1-c_{N'}-2 $ holds for all
level-1 networks $N'\in \cL_1(n-1)$. Clearly, $g(N)= n-c_N-2 $
holds in case $N$ is a phylogenetic tree as in that case $g(N)=0$
and every non-trivial cut arc of $N$ is an interior edge of $N$\steven{,}
of which there are $n-2$. So assume that
$N\in \cL_1(n)^-$. To see the stated
bound for $g(N)$, we distinguish between the cases
that (i) $N$ contains a gall $C$ whose
outgoing arcs are all trivial cut arcs and (ii) that this is not the case, 
that is, $N$ contains a cherry.

Assume first that Case (i) holds. We distinguish the
cases that $C$ has three outgoing arcs and that $C$ has
at least four outgoing arcs. Assume first that  $C$ has
at least four outgoing arcs. Then there must exist a
leaf $a$ of $N$ that is the head of an outgoing arc
of $C$ but is not adjacent with the unique hybrid vertex of $C$.
Consider the rooted DAG $N'$ obtained from $N$ by first
removing $a$, its parent $a'\in V(N)$, and all arcs
adjacent with $a'$ and then adding a new arc 
from the parent of $a'$ to the 
child of $a'$ contained in $V(C)$. Clearly,  
$N'$ is a binary level-1 network on $L(N) \setminus \{a\}$
and $g(N)=g(N')$ and $c_N=c_{N'}$ both hold.
Since $|L(N')|= n-1$, we have
$g(N)=g(N')=n-1 - c_{N'} -2=n-3 - c_N< n-c_N-2$, by \steven{the} induction hypothesis.
Consequently, $g(N)<n - c_N -2$ holds in
this case.

Next, assume that $C$ has exactly three outgoing arcs
$a_1,a_2,a_3$.
Let $N'$ be the rooted DAG obtained from $N$ by contracting
$C$ as well as $a_1$, $a_2$, and $a_3$ into a new leaf $x$.
 Clearly, $N'$ is a binary level-1 network on 
 $L(N)\cup\{x\}\setminus\{head(a_1),head(a_2), head(a_3)\}$
and $g(N')=g(N)-1$ and $c_{N'}=c_N-1$.
\steven{Thus, $g(N') \leq n-2 - c_{N'} -2$ and, so, 
$g(N) \leq n-c_N-2$ holds in this case too.}

\steven{Last but not least}, assume that Case (ii) holds, that is, $N$ contains two
leaves $x$ and $y$ that form a cherry. Let $N'$ denote the
rooted DAG obtained from $N$ by first deleting $x$, its
parent $p$, and all arcs incident with $p$ and then adding a new
arc from the parent of $p$ to $y$. Clearly $N'$ is a binary level-1 network
on $L(N)\setminus \{x\}$ and $g(N')=g(N)$ and $c_{N'}=c_N-1$ both hold.
Consequently, $g(N)=g(N') \leq n-1 - c_{N'} -2= n-1-(c_N-1)-2=n-c_N-2$
holds by \steven{the} induction hypothesis. This concludes the proof of this case and
thus the proof of the stated bound for $g(N)$.

It remains to establish that the stated bound for $g(N)$ 
is tight for a level-1
network $N\in \cL_1(n)$ for which
all of its galls have precisely three outgoing arcs. To see this\steven{,} one can 
again perform induction on $n\geq 2$ but this time assuming that 
$g(N')=n-c_{N'}-2$ holds 
for all level-1 networks $N'\in\cL_1(n-1)$ for which every gall
has precisely three outgoing arcs.
\steven{In this context} it should  be noted that the cases $n\in\{2,3\}$ 
and $N$ is a phylogenetic tree on $n$ leaves have already been 
observed above. We leave the details to the interested reader.
\end{proof}

\section{Counting clusters and triplets}
In this section we establish enumerative results for computing
the sizes of the so-called hardwired and softwired
cluster system, respectively, that have both been
introduced in the literature for phylogenetic network
reconstruction \textcolor{black}{\cite{HRS10}}. In addition, we  establish that
the corresponding result for triplets does not hold.
We start with clusters.

\subsection{Counting clusters}
\label{counting-softwired}
We call a non-empty subset of $X$ a {\em cluster} 
and refer to a set of clusters of $X$ as a {\em cluster system on $X$}, 
or just a cluster system if the set $X$ is clear from the context.
Suppose for the following that $N$ is a phylogenetic network on $X$
and that $v\in V(N)$. Then we define the cluster $C_N(v)$ associated \steven{with}
$v$ to comprise of
all leaves of $N$ that are below $v$ and \steven{let} $C_N(v)=\{v\}$
in case $v$ is a leaf of $N$. Again, we simplify our notation 
by writing $C(v)$ rather than $C_N(v)$ if $N$ is clear from the context.
Note that $C(\rho_N)=X$. Then the 
{\em hardwired cluster system} $\mathcal C(N)$ 
associated \steven{with} $N$ is the cluster system $\{C(v)\,:\, v\in V(N)\}$.
Note that if $N\in\cL_1(X)^-$\steven{,} then 
Lemma~\ref{lem:countbounds}
implies that $2n+1\leq |\mathcal C(N)|\leq 3n-2$ 
and if $N$ is \steven{a} phylogenetic tree\steven{,} then $|\mathcal C(N)|=|V(N)|=2n-1$\steven{,}
where $n$ denotes $|X|$ \steven{in both cases}.
 Denoting by $\mathcal T(N)$ the 
set of phylogenetic trees on $X$ 
displayed by $N$, the  {\em softwired cluster system} $\mathcal S(N)$
associated \steven{with} $N$ is defined as
$\mathcal S(N)=\bigcup_{T\in\mathcal T(N)} \mathcal C(T)$. 

To illustrate this definition\steven{,} consider the level-1 network $N_1$  
on $X=\{x_1,\ldots,x_5 \}$
depicted in Fig.~\ref{fig:softwired-clusters}(i). 
Then $\mathcal S(N_1)$
comprises the clusters $ X$, 
$\{x_2, x_3, x_4, x_5\}$, $\{x_3,x_4,x_5\}$, $\{x_4,x_5\}$, 
$\{x_2,x_3,x_4\}$, $\{x_3,x_4\}$, $\{x_1\}$, $\{x_2\}$,
 $\{x_3\}$, $\{x_4\}$, and $\{x_5\}$.

\begin{figure}[!ht]
\begin{tabular}{ccc}
\includegraphics[scale=0.4] {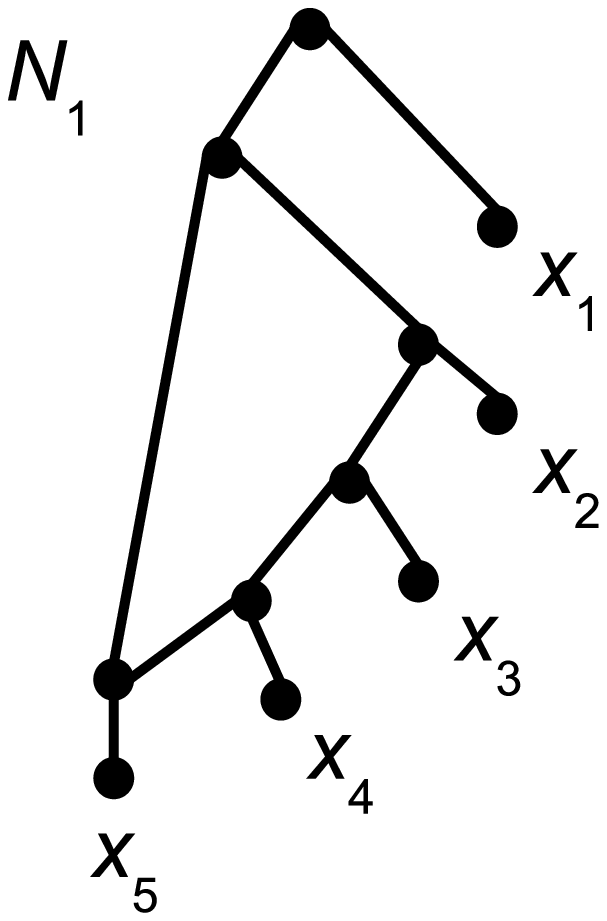} & 
~~~~~~~~~~~~~~~~ &
\includegraphics[scale=0.4] {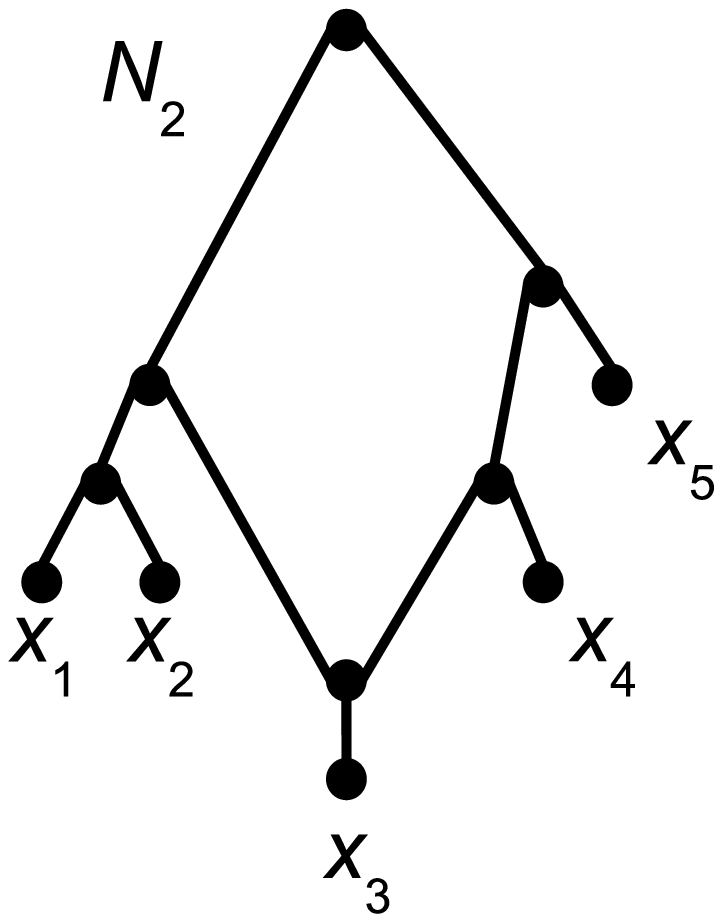}\\
(i) & ~ & (ii)
\end{tabular}
\caption{\label{fig:softwired-clusters} Two networks $N_1$ and $N_2$ on $X=\{x_1,\ldots, x_5\}$.
Note that both networks are 4-outwards, but neither is saturated.}
\end{figure}

\steven{It is not too difficult to argue that $\mathcal S(N)$ contains $O(n)$ clusters. To
see this, let $T$ be a tree on $X$ displayed by $N$ and let $v$ be a vertex of $T$. 
From the definition of display it follows that a subdivision of $T$ can be topologically
embedded within $N$. Fix such an embedding, and let $T'$ and $v'$ be the
images of $T$ and $v$ in $N$. If $v'$ is the head of a cut arc in $N$, or the root of $N$, then $C_T(v)$ will be equal to $C_{N}(v')$, irrespective of the exact embedding. If $v$ is not
the head of a cut arc, nor the root, then it is a vertex of some gall of $N$. In that case, there are
only (at most) two possibilities for $C_T(v)$. Specifically, the choice of cluster is completely determined by which of the two edges incoming
to the hybridization vertex in the gall, are in $T'$ (irrespective of the exact topology of the embedding). Now, from Lemma~\ref{lem:countbounds} $N$
contains $O(n)$ vertices. Given that (as argued) each vertex can contribute at most two clusters, it follows that $\mathcal S(N)$ contains $O(n)$ clusters.}
The next result improves on this $O(n)$ observation
by providing a formula for the size of the closely related
cluster system $\mathcal S(N)^-:=\mathcal S(N)\setminus\{X\}$.
\steven{(This is also an improvement on the result
presented in~\cite[Proposition 3]{GH12}.)} To establish it, we
require further terminology.
%

Suppose $N\in\cL_1(X)$ and
$X' \subseteq X$.
Then we define the {\em restriction $N|_{X'}$ of $N$
to $X'$} to be the
network in $\cL_1(X')$ obtained from $N$ 
by \textcolor{black}{deleting all vertices in $X - X'$ and then} applying the following \kelk{``cleaning up''} operations in any order until no
more can be applied\footnote{\kelk{In this paper only a subset of these ``cleaning up'' operations will be required. However, we list them all to retain consistency with the wider literature.}} :  (i) suppressing vertices with
in-degree and out-degree both equal to one; (ii) deleting vertices 
with out-degree zero that are not an element in $X$; (iii) collapsing
multi-arcs into a single arc; (iv) if a gall $G$ has
been created that has 
exactly two outgoing cut arcs 
$(u,v)$, $(u',v')$, then deleting these two cut arcs and all the arcs
of $G$ and adding arcs $(r, v)$ and $(r, v')$\steven{,} where
$r$ is the unique vertex of $G$ whose children are $u$ and $u'$; \kelk{(v) deleting vertices
with in-degree zero and out-degree one.} (Note that if $N$ is a phylogenetic tree this
definition specializes to the usual definition of ``restriction'' used in the tree literature.) We
often write $N|_{X-x}$ as shorthand for $N|_{X - \{x\}}$.


\begin{theorem}\label{theo:countingClusters}
Let $n\geq 2$ and suppose $N\in \cL_1(n)$.
 Then $|\mathcal S(N)^-|=3n - 4 - c_N$. 
\end{theorem}
\begin{proof}
We prove the theorem by induction on $n\geq 2$.
Suppose $N\in\cL_1(n)$.
Then the stated equality holds if $n=2$ as then
$N$ is a phylogenetic tree on two leaves and if $n=3$
because in that case 
$N$ is either simple and so $c_N=0$ holds or $N$ is a triplet.
In the former case, $|\mathcal T(N)|= 2$ and both phylogenetic trees 
contained in $\mathcal T(N)$ are triplets. Thus, 
$|\mathcal S(N)^-|=5=3n - 4 - c_N$
holds in this case. In the latter case, $c_N=1$ follows and thus 
$|\mathcal S(N)^-|=4=3n - 4 - c_N$ in this case, too.

Now suppose $n>3$ 
and assume that
\textcolor{black}{the theorem holds for all networks $N'$ with at most $n-1$ leaves}.
\steven{Let} $X=L(N)$.
We distinguish between
the cases that every cut arc of $N$ is trivial and the case that $N$ contains at least one non-trivial cut arc.

Suppose first that every cut arc of $N$ is trivial. Then $c_N = 0$ 
and $N$ is simple. Note that since $n>3$, at least one 
of the two directed paths  
from the root $\rho_N$ to the hybrid vertex $h_N$ of $N$ must 
contain at least two vertices distinct 
from $\rho_N$ and $h_N$. Let $P_1$ denote
such a path. Moreover, let $v\in V(P_1)$ denote the vertex on 
$P_1$ that is adjacent
with $\rho_N$  and let $x\in X$ denote the leaf of $N$ that is adjacent
with $v$. \steven{Let} $X'=X-\{x\}$ and $N'=N|_{X'}$.
%
Clearly $N'\in\cL_1(n-1)$ and
$c_{N'}=c_N=0$. Thus, 
$|\mathcal S(N')^-|=3(n-1) - 4$, by \steven{the} induction hypothesis. 
Observe that the definition of $S(N)$ implies that
$S(N)^{-}$ contains exactly three clusters that
$S(N')^{-}$ does not. Indeed, in case the other directed 
path from $\rho_N$ to $h_N$  also contains
vertices distinct from $\rho_N$ and $h_N$, the three clusters missing
from $S(N')^{-}$ are 
$\{x\}$, $C_N(v) \setminus \{h\}$ and $C_N(v)$, where $h$ is the leaf below
$h_{N}$. Otherwise, the 
three clusters missing from $S(N')^{-}$
 are $\{x\}$, $C_N(v) \setminus \{h\}$ and $C_{N}(v')$ where
$v'$ is the child of $v$ that is not contained in $X$.
Consequently, 
$|\mathcal S(N)^-|=|\mathcal S(N')^-| +3= 3(n-1) - 4-0+3=3n-4-c_N$
holds in this case.

Now suppose that $N$ has a non-trivial cut arc $e=(u,v)$. \steven{Let} 
$Y_1=\{l\in X: l \mbox{ is below } v\}$ and $Y_2=X-Y_1$. Note that 
$2\leq |Y_1|<n$. Hence, $1\leq |Y_2|\leq n-2$. 
Consider the rooted DAG $N_1$ with leaf set $Y_1$ obtained from $N$ by
deleting all vertices (plus their incident arcs) that are not below $v$
and the rooted DAG $N_2$ on $Y_2\cup\{v\}$
obtained from $N$ by deleting 
all vertices below $v$ (plus their incident arcs).
Since $|Y_1|\geq 2$ it follows that
$N_1\in\cL_1(Y_1)$
and since $|Y_2\cup\{v\}|\geq 2$, we have that $N_2\in\cL_1(Y_2\cup\{v\})$.
Note that a phylogenetic tree $T$ is displayed by $N$
if and only if $T|_{Y_1}$ is displayed by $N_1$ and $T|_{Y_2\cup\{v\}}$
is displayed by $N_2$. Consequently,
$\mathcal S(N)^-=\mathcal S(N_1)^-
\stackrel{\cdot}{\cup}\{C\in \mathcal S(N_2)^-\,:\, v\not \in C\}
\stackrel{\cdot}{\cup}\{C-\{v\}\cup Y_1\,:\, v\in C \mbox{ and } 
C\in \mathcal S(N_2)^-\}$ must hold. Thus, 
 $$
 |\mathcal S(N)^-|=|\mathcal S(N_1)^-|+ |\mathcal S(N_2)^-|.
 $$
 
Let $i=1,2$ and \steven{let} $n_i=|L(N_i)|$ and $c_i=c_{N_i}$. Then
$|\mathcal S(N_i)^-|=3n_i - 4 - c_i$, by \steven{the} induction hypothesis. Consequently,
$|\mathcal S(N)^-|=3n_1 - 4 -c_1 + 3n_2 - 4 - c_2$. Since
$n_1 + n_2 = n + 1$ and
$c_1 + c_2 = c_N - 1$ it follows that
$
|\mathcal S(N)^-| =3(n + 1) - 8 - (c_N - 1) = 3n - 4 - c_N,
$
holds in this case, too.
\end{proof}


\subsection{Counting triplets}\label{counting-triplets}
In view of the close relationship 
between the cluster system $\mathcal C(T)$ induced by a  phylogenetic
tree $T$ on at least three leaves 
and the triplet system $\mathcal R(T)$ induced by $T$
(see e.\,g.\, \cite{DHKMS12} or  \cite{vIK11b})
it is reasonable to hope that the 
companion result to Theorem~\ref{theo:countingClusters}
might hold for the triplet system $\mathcal R(N)$ induced by a phylogenetic
network $N$ on at least three leaves. Put differently,
it should be possible to express the size of $\mathcal R(N)$
in terms of the number of galls and non-trivial cut arcs of $N$. As the
next result \steven{shows}, this is not the case. We start with defining the
triplet system $\mathcal R(N)$.


Suppose $N\in\cL_1(X)$\steven{,} where $|X|\geq 3$ and
$a$, $b$, and $c$ are distinct elements in $X$. Then the
triplet $ab|c$ is said to be {\em consistent} with $N$ 
if there exist distinct vertices
$v$ and $w$ in $N$ and directed paths in $N$ from $v$ to $c$ and $w$,
respectively, and from $w$ to $a$ and $b$, respectively, such that any pair
of those paths does not have an interior vertex in common.
The triplet system $\mathcal R(N)$ is then the set of all triplets $t$
with $L(t)\subseteq X$ that are consistent with $N$.

 To illustrate this definition consider the simple level-1 network $N_2$  
on $X=\{x_1,\ldots,x_5\}$
depicted in Fig.~\ref{fig:simple}(ii). Then 
$\mathcal R(N_2)$ comprises the sixteen triplets $x_3|x_1x_2$, $x_4|x_1x_2$,
 $x_5|x_1x_2$, $x_1|x_3x_4$, $x_4|x_1x_3$, $x_1|x_3x_5$, 
 $x_5|x_1x_3$, $x_1|x_4x_5$, $x_2|x_3x_4$, $x_4|x_2x_3$,
  $x_2|x_3x_5$, $x_5|x_2x_3$, $x_2|x_4x_5$,
$x_3|x_4x_5$, $x_5|x_3x_4$, $x_1|x_2x_3$.


\begin{proposition}\label{prop:triplet-system-size}
For all $n\geq 6$\steven{,} there exist distinct 
networks $N_1,N_2\in \cL_1(n)$ with 
the same number of galls
and non-trivial cut arcs but $|\mathcal R(N_1)|\not=|\mathcal R(N_2)|$.
\end{proposition}
\begin{proof}
Let $X'$ denote a finite set of size at least two
and let $a$, $b$, $c$, and $d$ denote pairwise 
distinct elements not contained in $X'$.
Consider the binary level-1 networks $N_1$ and $N_2$ on 
$X:=X'\cup\{a,b,c,d\}$ depicted in 
Fig.~\ref{FigTripletSystemSize}\steven{,}
where the triangle marked $T$ denotes some binary phylogenetic 
tree on $X'$.

As can be easily checked, $N_1$ and $N_2$ have the 
same number of leaves and both contain one gall and have 
$c_{T}+3$ non-trivial cut arcs.
Moreover, for any $3$-set $Y\in {X\choose 3}$\steven{,} there
exists exactly one triplet on $Y$ that is contained in $\cR(N_1)$ 
except for $Y=\{a,b,c\}$ for which 
$a|bc, c|ab\in  \cR(N_1)$ holds. Hence,
 $|\cR(N_1)|=\binom{n}{3}+1$\steven{,} where $n=|X|$.  Similarly
 for every 3-subset $Y\in {X\choose 3}$\steven{,} there
exists exactly one triplet on $Y$ that is contained in $\cR(N_2)$ 
 except for $Y=\{a,c,x\}$ with $x \in X'\cup \{b\}$ for which
$a|cx,c|ax\in\cR(N_2)$ holds. Consequently,
$|\cR(N_2)|=\binom{n}{3}+1+|X'|>\binom{n}{3}+1= |\cR(N_1)|$.

\begin{figure}[!ht]
\begin{tabular}{cc}
\includegraphics[scale=0.4] {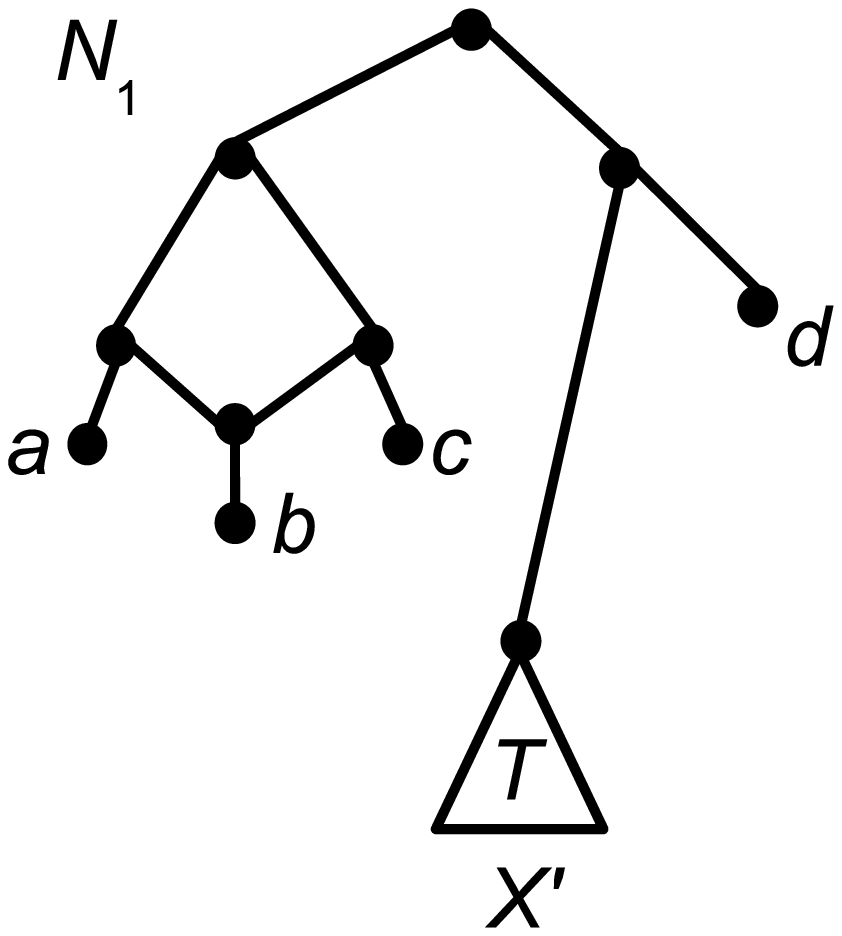} & 
\includegraphics[scale=0.4] {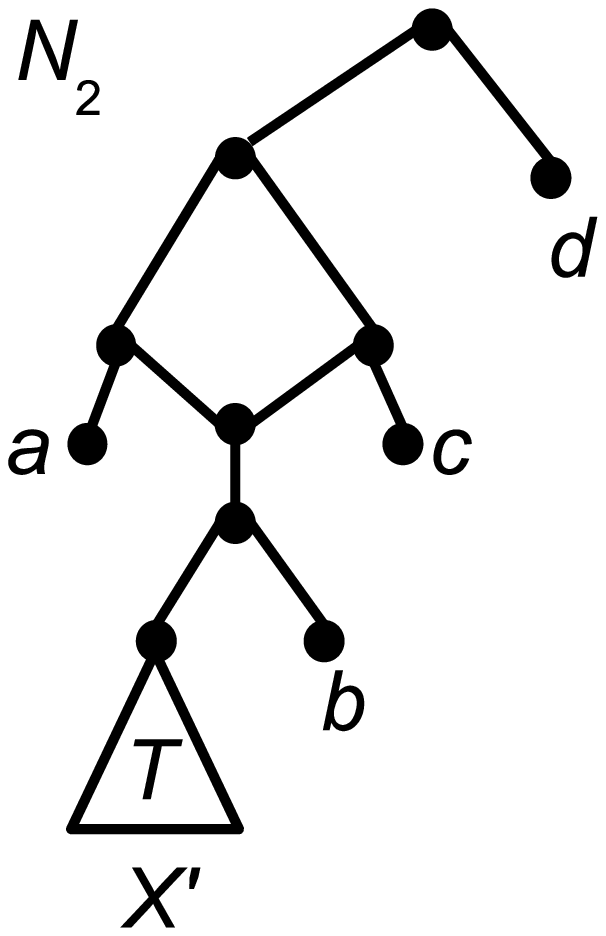}\\
\end{tabular}
\caption{Two binary level-1 networks $N_1$ and $N_2$ on 
$X'\cup\{a,b,c,d\}$ for which the respective 
number of leaves, galls, and non-trivial cut arcs are the same yet 
$|\cR(N_2)|\not=|\cR(N_1)|$ -- see the proof of 
Proposition~\ref{prop:triplet-system-size} for details.}
\label{FigTripletSystemSize}
\end{figure}
\end{proof}

\section{Triplet systems and the partition $Cut(N)$}
\label{sec:triplets-sn}

In this section, we start turning our attention to the question of how many
triplets suffice to uniquely determine a binary level-1 network. 
Central to our arguments will be a special type of subsets of $X$ called
SN-sets which were originally introduced in \cite{JNS06} and further studied in,
for example, \cite{IKKSHB09,vIK11a}.

Suppose $|X|\geq 3$ and $\mathcal R$ is a triplet system on $X$. Then a subset
$S\subseteq X$ is called an {\em SN-set} of $\mathcal R$ if there is no triplet 
$xy|z\in\mathcal R$ with $x, z \in S$ and $y \not \in S$. In addition, such a 
set $S$ is called {\em non-trivial} if $S\not=X$.\footnote{\steven{There is some confusion in earlier literature whether $\emptyset$ should be
considered an SN-set. Here we allow this, as it does not adversely affect our analysis. Although it sounds a little strange, $\emptyset$ is also a non-trivial SN-set. We adopt this convention to ensure consistency with earlier publications: $X$ is the only trivial SN-set.}} \steven{Last but not least}, a \textcolor{black}{non-trivial} SN-set
$S$ for $\mathcal R$ is called {\em maximal} if there is no 
non-trivial SN-set that is 
a strict superset of $S$.

As it turns out, for a binary network $N$ (of any level) 
the SN-sets of $\mathcal R(N)$ are closely related to the cut arcs of $N$
in the sense that if $(u,v)$ is a cut arc of $N$, then 
$C_N(v)$ is an SN-set of $\mathcal R(N)$ because there cannot 
exist a triplet $xy|z \in \mathcal R(N)$ such that $x,z\in C_N(v)$ 
and $y\not\in C_N(v)$. \kelk{We call a cut arc $(u,v)$ of $N$ \emph{highest} 
if there does not exist a cut arc $(u', v')$ of $N$ such that there 
is a directed path from $v'$ to $u$.} \kelk{We denote by $Cut(N)$ the partition 
of $X$ induced by, for each highest cut arc $(u,v)$ of $N$, taking the 
cluster $C_N(v)$ of $X$.} \kelk{By \cite[Observation 3]{vIK11a}, 
$Cut(N)$ is exactly the set of maximal SN-sets of $\mathcal R(N)$.}

 To illustrate, consider the network $N_1$ on 
$X=X'\cup\{a,b,c,d\}$ depicted in 
Fig.~\ref{FigTripletSystemSize}.
Then $Cut(N_1)$ is the bipartition $\{\{a,b,c\},X' \cup \{d\}\}$.

We begin with an auxiliary observation which relies on the concept of
compatibility of pairs of sets, whereby two non-empty finite sets $A$ and $B$
are called {\em compatible} if $A\cap B\in \{\emptyset,A,B\}$ holds
and {\em incompatible} otherwise. More generally, a cluster system $\mathcal C$ 
is called {\em compatible} if any two clusters in $\mathcal C$ are compatible
and {\em incompatible} otherwise
(see e.\,g.\,\cite[Section 3.5]{SS03} and \cite{DHKMS12}
for more on such objects).

\begin{observation}
\label{obs:twocut}
Suppose that $n\geq 3$ and that $N$ and $N'$ are two networks in 
$\cL_1(n)$ such
that $\mathcal R(N) \subseteq \mathcal R(N')$. Let $v\in V(N)$
and $v'\in V(N')$ denote two \steven{split} vertices that are heads of cut arcs 
of $N$ and $N'$, respectively. Then the induced clusters
$C_N(v)$ and $C_{N'}(v')$ are compatible.
In particular, if $C_{N}(v) \subsetneq C_{N'}(v')$ then
$C_N(v)$ is not a maximal SN-set for $\mathcal R(N)$.
\end{observation}
\begin{proof}
\steven{Let} $C=C_N(v)$ and $C'=C_{N'}(v')$. 
Clearly, if $C=C'$ then $C$ and $C'$ are compatible.
So suppose $C\not=C'$. Assume for the sake of contradiction that
$C$ and $C'$ are not compatible, that is, 
$C\cap C'\not\in \{\emptyset,C,C'\}$. Choose elements
$x \in C \setminus C'$, $y \in C \cap C'$ 
and $z \in C' \setminus C$. Then, out of the three possible
triplets with leaf set $\{x,y,z\}$, only the triplet 
 $xy|z$ can be contained in $\mathcal R(N)$. Hence, $xy|z\in\mathcal R(N')$
and, so, $C'$ cannot be an SN-set of $\mathcal R(N')$; a contradiction
as the incoming arc of $v'$ is a cut arc of $N'$ and, so, 
$C'$ must be an SN-set of 
$\mathcal R(N')$. Thus, $C$ and $C'$ must be compatible.

To see the remainder of the observation, assume that
$C\subsetneq C'$. Then since 
$\mathcal R(N) \subseteq \mathcal R(N')$ and $C'$ is an
SN-set of $\mathcal R(N')$, it follows that 
$C'$ is also an SN-set of $\mathcal R(N)$. Since $C$ is also an 
SN-set of $\mathcal R(N)$,
it cannot be a maximal SN-set of $\mathcal R(N)$.
\end{proof}

The next result will be required by the induction argument 
that we will use in the proof of Theorem~\ref{theo:shapely-saturated-triplet}.
\kelk{The proof of the proposition}
relies on the facts that 
for any saturated network $N\in\cL_1(X)$ (i) the partition $Cut(N)$ contains
at least three elements and (ii) there exists a gall $B$ of $N$
such that the root of $N$ is a vertex of $B$.

\begin{figure}[!ht]
\begin{tabular}{cc}
\includegraphics[scale=0.4] {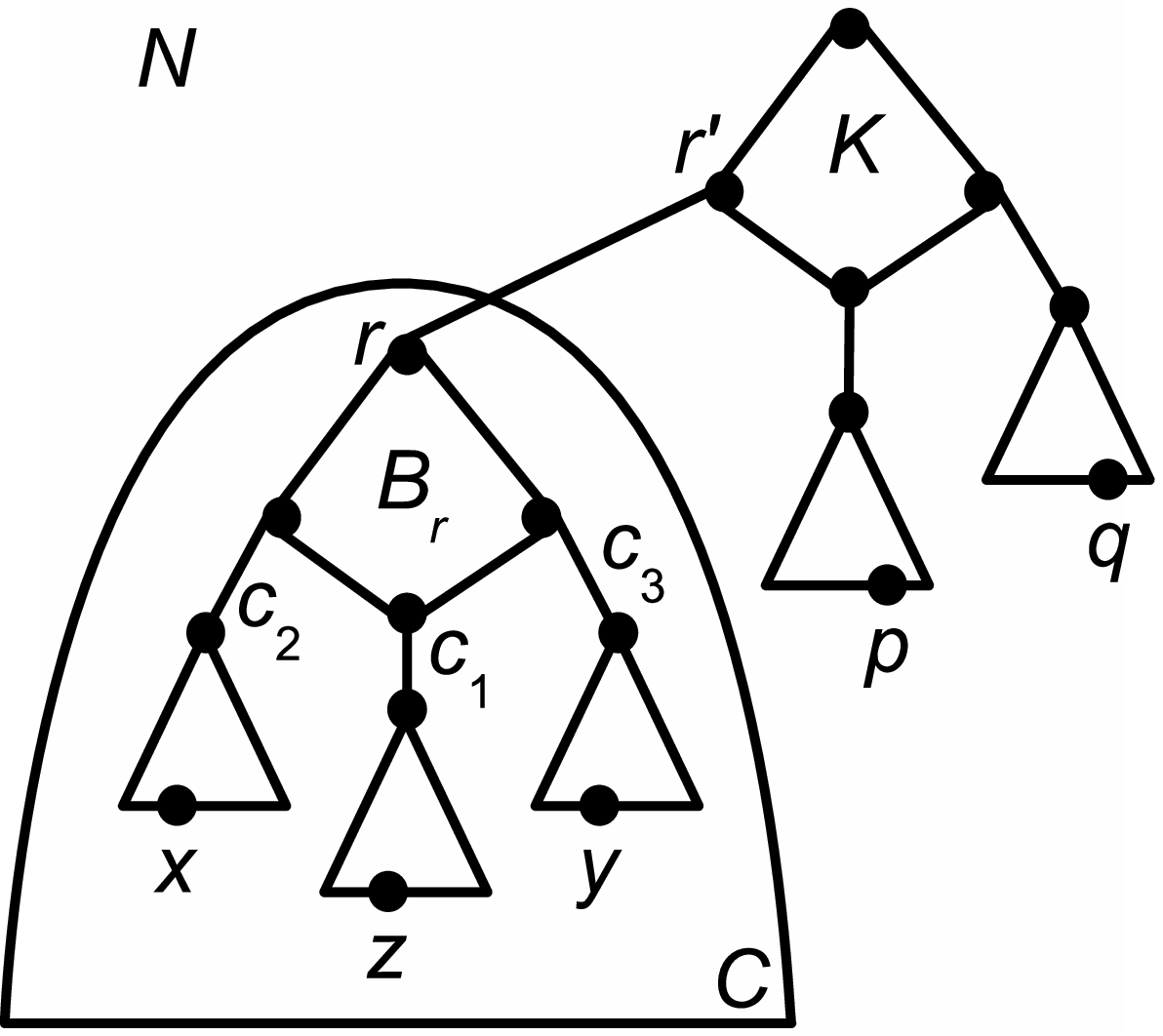} & 
\hspace{1cm}
\includegraphics[scale=0.4] {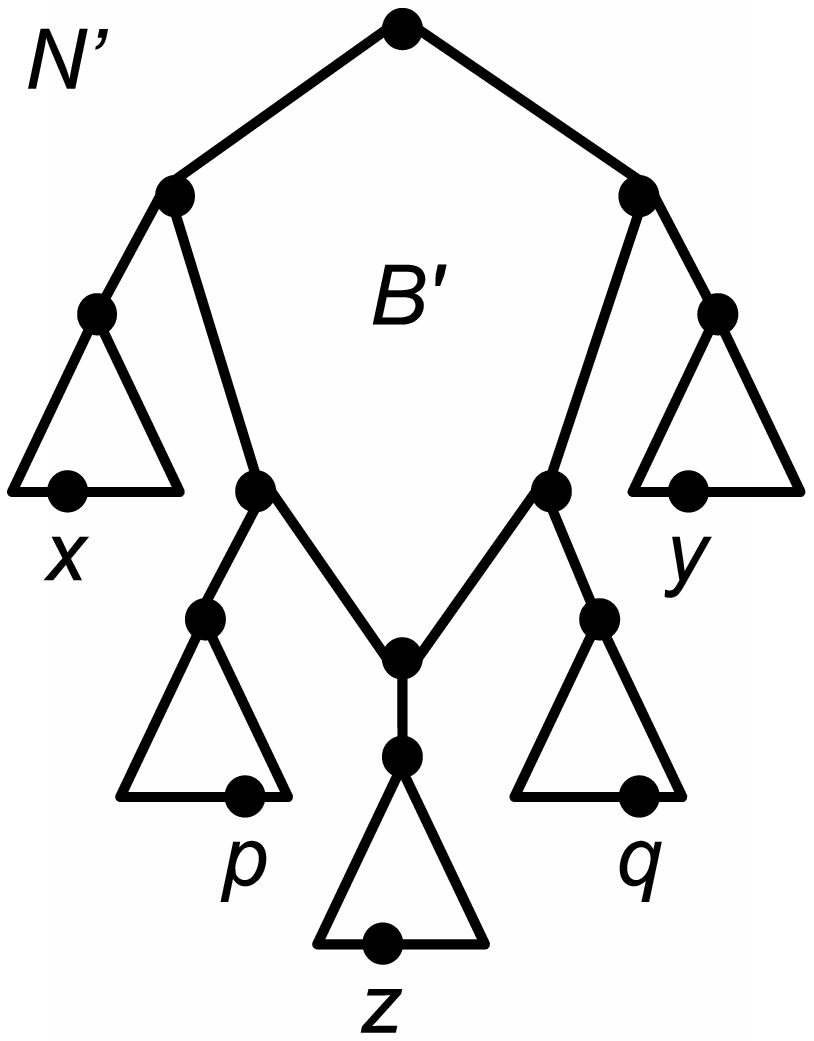}\\
(i) & (ii)
\end{tabular}
\caption{The structure of networks \steven{(i)} $N$ and \steven{(ii)} $N'$ considered in the
proof of Proposition \ref{prop:samecut}.}
\label{fig:prop51}
\end{figure}

\begin{proposition}
\label{prop:samecut}
Suppose  that $|X|\geq 3$, that  
$N$ is saturated network in $\cL_1(X)$ and that $N'$ is
a network in $\cL_1(X)$ 
such that $\mathcal R(N) \subseteq \mathcal R(N')$. Then $Cut(N)=Cut(N')$.
\end{proposition}
\begin{proof}
\steven{The proof contains multiple parts so we first describe it at a high level, and then give details. The entire proof is devoted to proving that $Cut(N)\subseteq Cut(N')$, and after this $Cut(N) = Cut(N')$ follows
immediately from the fact that $Cut(N)$ and $Cut(N')$ are both partitions of $X$. The proof that
$Cut(N)\subseteq Cut(N')$ holds is a long proof by contradiction, which starts with the
assumption that there exists some  $C\in Cut(N)-Cut(N')$. 
We then show that $|C|\geq 2$
must hold. Combined with the assumption that $N$ is saturated
we then infer that, up to symmetry, the structure of $N$
is as indicated in Fig.~\ref{fig:prop51}(i). Choosing elements
$x,p,q\in X$ as described below we obtain that  $\mathcal R(N)$
must contain two distinct triplets $t_1$ and $t_2$ with leaf set
$\{x,p,q\}$. By examining the structure of $Cut(N')$ we
identify that at least one of two cases, referred to as Cases (i) and (ii) below,
must hold. However, we show that Case (i) cannot hold, and thus conclude
that Case (ii) must hold. We then show that, up to symmetry,
the structure of $N'$ is as indicated  in Fig.~\ref{fig:prop51}(ii). We argue that $x, p, q$ are below three distinct highest cut arcs of $N'$, and that none of these are
the cut arc incident to the hybridization vertex of $B'$ (where $B'$ is the topmost gall of $N'$, as shown in
Fig.~\ref{fig:prop51}(ii)).
This implies that 
$\mathcal R(N')$ cannot contain both $t_1$ and $t_2$
which finally yields the required contradiction.\\
\\
Let us then start by assuming, for the sake of contradiction, that
there exists some  $C\in Cut(N)-Cut(N')$.}\\
\\
{\em \steven{Proof that $|C|\geq 2$:}} Since both $Cut(N)$
and $Cut(N')$ are partitions of $X$, there exists some $C'\in Cut(N')$
distinct from $C$ such that $C\cap C'\not=\emptyset$. 
Since, in view of \cite[Observation 3]{vIK11a}
recalled above, $C$ is a maximal SN-set of  $\mathcal R(N)$, 
\steven{and, by
Observation~\ref{obs:twocut}, $C$ and $C'$ are compatible, 
it follows that $C'\subsetneq C$.}
Thus, there exists a further
element $C''\in Cut(N')$ distinct from $C$ and $C'$ 
such that $C\cap C''\not=\emptyset$ holds, too. Thus, $|C|\geq 2$. \\

{\em \steven{The structure of $N$:}}
Let $r\in V(N)$ denote the head of the cut arc $(r',r)$ of $N$ for which 
$C=C_N(r)$ holds and let $B_r$ denote the gall of $N$ that contains $r$  
in its underlying cycle (which exists because $|C|\geq 2$ and $N$ is 
saturated). In view of the usual assumption that no gall in 
a phylogenetic network has two or fewer outgoing cut arcs, $B_r$ has 
at least three outgoing cut arcs $c_1, c_2$ and $c_3$ (see Fig.~\ref{fig:prop51}(i)). Let $c_1$ denote the outgoing cut arc of $B_r$ whose
tail is the hybrid vertex $h_{B_r}$ of $B_r$. 
Let $z\in C_N(h_{B_r})$, let
$x\in C_N( head(c_2) )$ and let $y\in C_N( head(c_3) )$. Clearly, $\mathcal R(N)$  
contains two distinct triplets $t$ and $t'$ on $\{x,y,z\}$. \\

{\em \steven{The structure of $N'$:}}
Since $\mathcal R(N)\subseteq \mathcal R(N')$ we also have
$t,t'\in \mathcal R(N')$.
Since $Cut(N')$ is the partition of $X$  induced by the maximal SN-sets of 
$\mathcal R(N')$, it follows that either (i) there exists some element 
$A\in Cut(N')$ such that $x,y,z\in A$ or (ii) there exist distinct
elements $C_x,C_y,C_z\in Cut(N')$ such that $a\in C_a$, for all
$a\in \{x,y,z\}$. 


Assume first that Case (i) holds. We claim that $C \subseteq A$. 
To see this, note that we were free to choose any two cut arcs $c_2$ and $c_3$ 
distinct from $c_1$ and subsequently we had a free choice of $z$, $x$, $y$. 
For any $Z:=\{x,y,z\}$ chosen this way - let us call this a valid choice - 
it is straightforward to see that there exist two triplets on $Z$ in
$\mathcal{R}(N)$ and thus in $\mathcal{R}(N')$. Since $A$ is an SN-set of 
$\mathcal{R}(N')$ it follows that 
as soon as two of the three leaves of a triplet on $Z$ are contained in $A$, 
so too is the third. 
Now, let $\{x,y,z\}$ be our initial valid choice, so by assumption 
$\{x,y,z\} \subseteq A$. Simple case analysis shows that for any 
element $p\in C$, 
at least one of $\{x,y,p\}$, $\{x,p,z\}$ or $\{p,y,z\}$ is a valid choice. 
Hence, $p\in A$ which proves the claim. Since $C \not \in Cut(N')$
we have  in fact $C\not=A$. But $C\subsetneq A$ cannot hold either
because $C$ is a maximal SN-set for $\mathcal{R}(N)$ and $A$ is a 
maximal SN-set for $\mathcal{R}(N')$, and
by Observation \ref{obs:twocut} this cannot happen. Thus, Case~(ii) 
must hold (see Fig.~\ref{fig:prop51}(ii)).\\

{\em \steven{The triplets $t_1$ and $t_2$:}}
Let $h\in V(N)$ denote the hybrid vertex of the topmost gall $K$ of $N$,
that is, the gall of $N$ that contains the root of $N$ in its vertex set
(which must exist because $N$ is saturated). Also note that 
because $C\in Cut(N)$ it follows that $(r',r)$ is a highest cut arc of
$N$ and thus $r'$ is a vertex of $K$. Since $|Cut(N)|\geq 3$
there exist distinct elements $C_1,C_2\in Cut(N)-\{C\}$
such that $C_N(h)\in \{C,C_1,C_2\}$. Choose some $p\in C_1$
and some $q\in C_2$. Combined with the definition of $\mathcal R(N)$
it follows that $\mathcal R(N)$ must contain two triplets $t_1$
and $t_2$ on $\{x,p,q\}$,
two triplets on $\{y,p,q\}$, and two triplets
on $\{z,p,q\}$. Note that since $\mathcal R(N)\subseteq \mathcal R(N')$,
those six triplets are also contained in 
$\mathcal R(N')$. \steven{(In the next part of the proof we assume $C_{N'}(h') = C_z$, where these terms will be defined in due course, and the critical point
here is that  $C_{N'}(h') \neq C_x.$ This is genuinely without loss of generality because
when selecting $t_1$ and $t_2$ in the present part of the proof it does not matter whether they are on $\{x,p,q\}$, $\{y,p,q\}$ or
$\{z,p,q\}$.)}\\

{\em \steven{The taxa $x, p, q$ are all beneath distinct highest cut arcs of $N'$, but none of these
are incident to the hybridization vertex:}}
Using $x$, $y$, $z$, $p$ and $q$, we next analyze the structure 
of $Cut(N')$ (see Fig.~\ref{fig:prop51}(ii)).
Observe first that since $|Cut(N')| \geq 3$,
the root of $N'$ must be contained in a gall $B'$ of $N'$. Let $h'\in V(N')$ denote
the unique hybrid vertex of $B'$.
Let $C_p,C_q\in Cut(N')$ be such that $p\in C_p$ and $q\in C_q$.

We claim that $C_p$ and $C_q$ are distinct elements
in $Cut(N')-\{C_x,C_y,C_z\}$.
To see this, note first that, 
since the sets $C_x$, $C_y$ and $C_z$ are pairwise distinct
and $t,t'\in \mathcal R(N')$, it follows
that one of $x$, $y$, and $z$ must be contained in $C_{N'}(h')$.
Without loss of generality, assume that $z\in C_{N'}(h') = C_z$.  Note next that $C_p \neq C_q$. Indeed, at least two elements of $\{x,y,z\}$ 
are not contained in $C_p$, because $C_x, C_y$ and $C_z$ are distinct. Suppose, 
without loss of generality, $x\not\in C_p$. If $C_p = C_q$, then only the 
triplet $pq|x$ will be contained in $\mathcal{R}(N')$, 
contradicting the fact that $t_1$ and $t_2$ are distinct triplets
on $\{x,p,q\}$ contained in $\mathcal{R}(N')$. 
It remains to show that $C_p,C_q\not\in \{C_x,C_y,C_z\}$. 
Assume for \steven{the sake of} contradiction that $p\in C_x$.
Then only $xp|q$ is in $\mathcal{R}(N')$, because $q \not \in C_x$, 
contradicting the fact that both $t_1$ and $t_2$ are in $\mathcal{R}(N')$. 
Similarly, if $p$ is in $C_y$, then at most one of the two triplets on 
$\{y,p,q\}$ are in $\mathcal{R}(N')$, and if $p$ is in $C_z$, at most one of 
the two triplets on $\{z,p,q\}$ are in $\mathcal{R}(N')$.  So 
$C_p \not \in \{C_x, C_y, C_z\}$. By a symmetrical argument, 
$C_q \not \in \{C_x, C_y, C_z\}$. This proves the claim.\\
\\
\steven{By the previous claim, neither $p$ nor $q$ is in $C_z$. 
Since $x\not\in C_z$ it follows that
$t_1$ and $t_2$ cannot both be contained in $\mathcal{R}(N')$ which gives the 
final contradiction. Thus,
$Cut(N)\subseteq Cut(N')$, as required.\\ 
\\
Since both $Cut(N)$ and $Cut(N')$
are partitions of $X$, it follows that  $Cut(N)=Cut(N')$}.
\end{proof}

\section{Triplet systems that $\cL_1(X)$-define}
\label{sec:triplets-define}

As is well-known, every binary phylogenetic tree $T$ on $X$ is defined by the
triplet set $\mathcal R(T)$ induced by $T$\steven{,} where a
a phylogenetic tree $S$ on $X$ is said to be {\em defined} 
by a triplet system $\mathcal R$ on
$X$, if, up to equivalence, $S$ is the unique phylogenetic tree on
$X$ for which $\mathcal R\subseteq \mathcal R(S)$ holds 
(see e.\,g.\,\cite{SS03}). In this context
it is important to note that this uniqueness only holds within
the space of phylogenetic trees because all networks $N\in\cL_1(X)$
that display $T$ have the property that 
$\mathcal R(T) \subseteq \mathcal R(N)$. Combined with the
fact that the network $N$ pictured in Fig.~\ref{fig:new}(i)  is, up to 
equivalence, the only binary level-1 network on $X=\{x_1,\ldots, x_4\}$
that is consistent with all five triplets
depicted in  Fig.~\ref{fig:new}(ii) - a simple case analysis can be applied to verify this - and $|\mathcal R(N)|= 7$, it is 
natural to ask how many triplets suffice to ``uniquely determine'' 
a level-1 network. In this section we provide a partial 
answer to this question.
\begin{figure}[!ht]
\begin{tabular}{cc}
\includegraphics[scale=0.4] {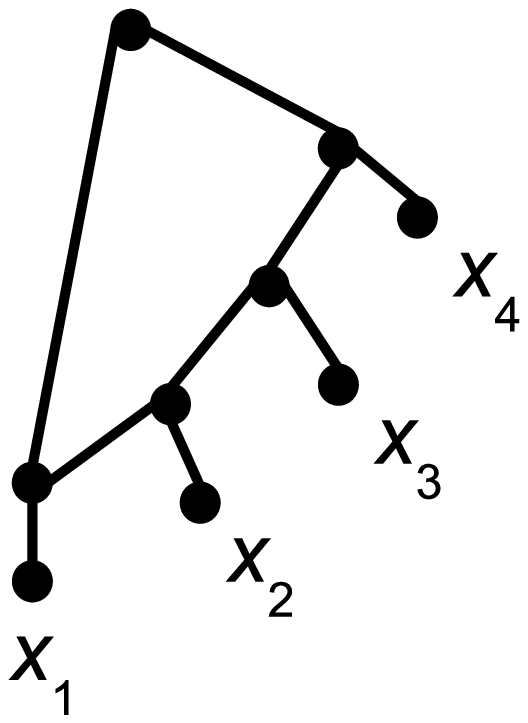} & 
\includegraphics[scale=0.4] {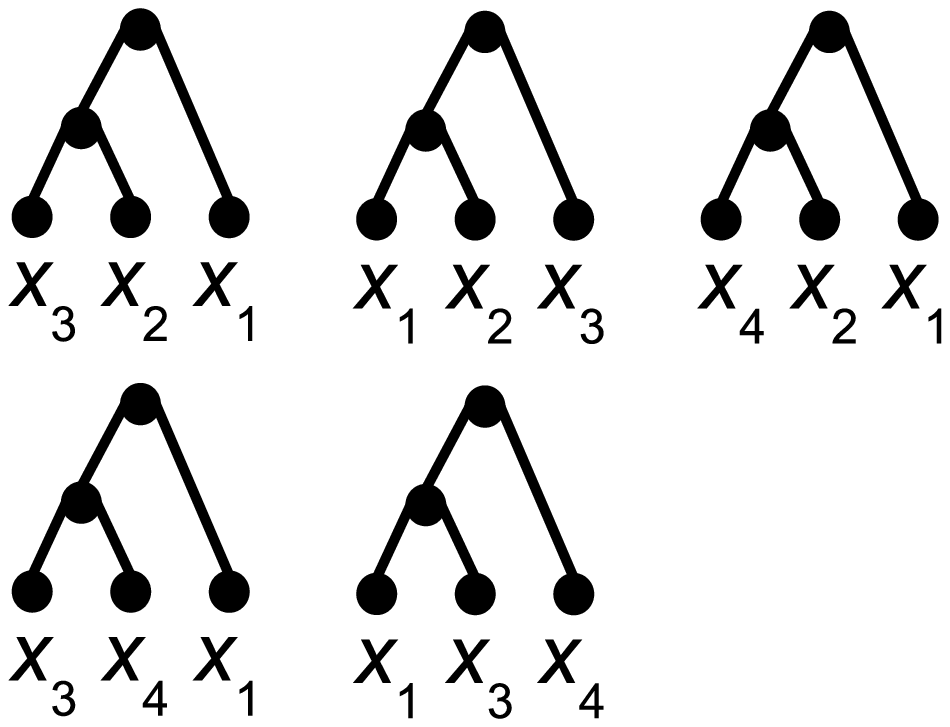}\\
(i) & (ii)
\end{tabular}
\caption{The binary level-1 $N$ on $X=\{x_1,\ldots, x_4\}$
depicted in (i) is uniquely determined by the five triplets
pictured in (ii) but $|\mathcal R(N)|=7$.}
\label{fig:new}
\end{figure}
More precisely, 
saying that  a network $N\in \cL_1(X)$ is {\em $\cL_1(X)$-defined
by a triplet system $\mathcal R$ (on $X$)} if, up to equivalence, 
$N$ is the unique network in $\cL_1(X)$ such that
$\mathcal R\subseteq \mathcal R(N)$ holds, we
 show that every 4-outwards network $N$ in $\cL_1(X)$ that is also
simple is $\cL_1(X)$-defined by a triplet system of size at most $2|X|-1$.
In addition, we show that if the requirement that
$N$ is simple is replaced by the requirement that $N$ is saturated\steven{,}
then $N$ is  $\cL_1(X)$-defined by 
$\mathcal R(N)$. 

\textcolor{black}{We note that 4-outwards is certainly a necessary condition for a network in $N\in \cL_1(X)$ to be $\cL_1(X)$-defined
by any triplet system. In particular, if a network $N$ has a gall with exactly three outgoing cut arcs then these can be permuted
without affecting $\mathcal{R}(N)$. However, we shall see later that 4-outwards is not, in isolation, a sufficient condition}.
 

As the next result \steven{shows} not all triplets
in $\mathcal R(N)$ are required to  $\cL_1(X)$-define a network
$N\in \cL_1(X)$ in case \steven{$N$ is not only 4-outwards but also simple.}
To simplify its exposition, we say that a triplet system
 on $X$ {\em $\cL_1(X)$-defines} a network $N\in \cL_1(X)$
if $N$ is {\em $\cL_1(X)$-defined} by it.

\begin{theorem}\label{theo:defineSimpleTriplets}
Every simple network in $\cL_1(X)$ with at least four leaves 
is $\cL_1(X)$-defined
by a triplet system of size at most $2|X|-1$.
\end{theorem}
\begin{proof}
We prove the theorem by induction on $|X|\geq 4$. 
Suppose $N$ is a simple network in $\cL_1(X)$\steven{,} where $n=|X|\geq 4$.
\steven{Let} $X=\{x_1,\ldots, x_n\}$.
Assume without loss of generality that $x_1$ is the head of the
outgoing arc of the unique gall $C$ of $N$ starting at the
hybrid vertex $v_1$ of $C$.
If $n=4$ then a straightforward case analysis
implies that  $N$ is $\cL_1(X)$-defined
by $\mathcal R(N)$. Note that $|\mathcal R(N)|=7=2n-1$ holds in this 
case.

Now assume that $n\geq 5$ and that for every set
$Y$ with $4 \leq |Y|\leq n-1$ and every 
simple network $N'\in\cL_1(Y)$
there exists a triplet system $\mathcal R$ of size at most $2|Y|-1$
such that $N'$ is $\cL_1(Y)$-defined by it. Starting at $v_1$ and
traversing the unique cycle $C$ in the underlying graph $U(N)$ of
$N$ counter-clockwise let $v_1,v_2,\ldots, v_{i-1},v_i=\rho_N,
v_{i+1},\ldots, v_{n+1},v_1$ denote a circular ordering of the vertices of
$C$. Without loss of generality assume that for all $2\leq j\leq i-1$
the head of the outgoing arc of $C$  
starting at $v_i$ is $x_i$ and that for
all $i+1\leq j\leq n+1$ the head of the outgoing arc of $C$ starting 
at $v_j$ is $x_{j-1}$. \steven{Let} $X'=X-\{x_n\}$.
We distinguish between the
cases that (i) the root $\rho_N$ of $N$ equals $v_{n+1}$ and
(ii) that this is not the case.

Case~(i): Assume that $\rho_N=v_{n+1}$ and \steven{let} 
$N'=N|_{X'}$. 
Since $N'$ is
clearly simple and $4\leq |L(N')|= n-1$ it follows by \steven{the} induction
hypothesis that there exists a triplet system $\mathcal R'$ on $X'$
of size at most $2(n-1)-1$ such that $N'$ is $\cL_1(X')$-defined 
by $\mathcal R'$. \steven{Let} $t_1=x_1|x_{n-1}x_n$ and $t_2=x_n|x_1x_{n-1}$.
We claim that $N$ is $\cL_1(X)$-defined by 
$\mathcal R=\mathcal R'\cup\{t_1,t_2\}$. To see this, assume that
$N_1$ is a network in $\cL_1(X)$ for which 
$\mathcal R\subseteq \mathcal R(N_1)$ holds. We need to show 
that $N$ and $N_1$ are equivalent.

\steven{Let} $N_1'=N_1|_{X'}$. By construction,
$\mathcal R'\subseteq \mathcal R(N_1')$. Since $N'$ is $\cL_1(X')$-defined 
by $\mathcal R'$ it follows that $N'$ and $N_1'$ must be equivalent. 
Consequently 
$N_1'$ must also be a simple network in $\cL_1(X')$.
Combined with the fact that $t_1,t_2\in\mathcal R\subseteq \mathcal R(N_1)$
it follows that $N$ and $N_1$ must be equivalent. \steven{Indeed, 
let $w$ denote the parent of $x_n$ in $N_1$. Then
$t_2$ implies that $w$ is either the head of one of the
two arcs of $N_1$ starting at the root $\rho'$ of the unique cycle of $N_1$
or a child of the root of $N_1$. Since
 $t_1$ implies that the parent $w'$ of $x_{n-1}$  is
below $w$, it follows  that $w$ must lie on the path in $N_1$
from $\rho'$ to $w'$.} 

Case~(ii) Assume that $\rho_N\not= v_{n+1}$. Then $i\in \{2,\ldots,n\}$.
We distinguish between the cases that $i=n$, that is, $\rho_N=v_n$
 and that $i\in \{2,\ldots,v_{n-1}\}$. In the former case 
the proof of the induction step is similar to the previous case
but with $t_1$ replaced by $x_{n-1}|x_nx_1$. In the latter case 
the proof of that step is also similar to the previous case
but this time with $t_2$ replaced by $x_{n-1}|x_nx_1$.
\end{proof}


Combined with the definition of  $\cL_1(X)$-defining
triplet systems, Theorem~\ref{theo:defineSimpleTriplets}
immediately implies:

\begin{corollary}\label{cor:defineSimpleTriplets}
Every simple network in $\cL_1(X)$ with at
least four leaves is $\cL_1(X)$-defined by
its induced triplet system.
\end{corollary}

An obvious problem with extending Corollary~\ref{cor:defineSimpleTriplets}
to general networks in  $\cL_1(X)$
 is that 4-outwards networks can have tree-like
regions. \steven{Consider, for example, then situation when a} a 4-outwards network $N$ contains a directed
path of length 3 or more consisting solely of cut arcs. \steven{We} can then
transform it into a new network $N'$ in $\cL_1(X)$ for
which $\mathcal R(N)\subseteq \mathcal R(N')$ holds
by subdividing the first and last cut arc of that path by new vertices
$u$ and $v$, respectively, and adding a new arc $(u,v)$.
There are however
more subtle situations possible which do not require adding vertices
and arcs. Consider, for example, the two networks 
$N$ and $N'$ on $X=\{x_1,\ldots, x_5\}$ presented in 
Figures~\ref{fig:simple}(ii) and \ref{fig:softwired-clusters}(ii),
respectively. Then $\mathcal R(N')\subseteq  \mathcal R(N)$
holds but $N$ and $N'$ are clearly not equivalent. Thus\steven{,}
$N'$ is not $\cL_1(X)$-defined by $\mathcal R(N')$ (although
$N'$ is clearly encoded by $\mathcal R(N')$ as it is $4$-outwards
\cite{GH12}).
We therefore next turn our attention to \steven{identifying additional
conditions which allow 4-outwards networks to be $\cL_1(X)$-defined by
their induced triplet systems.}

To establish our next result  
(Theorem~\ref{theo:shapely-saturated-triplet}), 
we require a construction from \cite{vIK11a} that
allows us to associate a level-1 network $Collapse(N)$ 
to any level-1 network $N$ such that $Collapse(N)$ is either simple, or is a 
phylogenetic tree on two leaves. We next review this construction
for networks in $\cL_1(X)$. 

\steven{Let $N$ be a network in $\cL_1(X)$}. 
For each element $C\in Cut(N)$ choose some element $c_C\in C$ and
\steven{let} $X^*=\{c_C\,:\,C\in Cut(N)\}$. Note that $|X^*|\geq 2$, 
but if $N$ is
saturated $|X^*|\geq 3$ and if $N$ is saturated and 4-outwards $|X^*|\geq 4$.
\steven{Then the rooted DAG $Collapse(N)$ is obtained from $N$ as follows: for 
each highest cut arc $(u,v)$ of $N$,
replace the (directed) subgraph of $N$ containing $v$ and all vertices
below $v$ \steven{(we denote this subgraph $N_{x_v}$ for later use)} by the unique element $x_v$
in $X^*\cap C_N(v)$.} Clearly, if $|C_N(v)|\geq 2$ then 
$N_{x_v}$ is contained in $\cL_1(C_N(v))$ and is an isolated vertex otherwise.
That $Collapse(N)$ is a simple network in $\cL_1(X^*)$
or a phylogenetic tree on two leaves is clear. 
Let $\mathcal R_{Collapse(N)}$ denote
the triplet system on $X^*$
comprising all triplets $x_w | x_u x_v$
%
for which there \steven{exist} $x_1 \in C_N(w)$, $x_2 \in C_N(u)$ and $x_3 \in C_N(v)$  such that $x_1|x_2x_3\in\mathcal R(N)$. It is
 straightforward to see
that $\mathcal R(Collapse(N)) = \mathcal R_{Collapse(N)}$.

\begin{theorem}\label{theo:shapely-saturated-triplet}
Every 4-outwards network in $\cL_1(X)$ that is also
saturated is $\cL_1(X)$-defined by its induced triplet system.
\end{theorem}
\begin{proof}
We prove the theorem by induction on the number $g(N)$ of galls 
in a saturated, 4-outwards network $N\in\cL_1(X)$. Suppose $N$ is such
a network. \steven{Let} $g=g(N)$. 
Then since $|X|\geq 2$ and $N$ is saturated we have $g\geq 1$. 
Hence, $|X|\geq 3$.
In case $g=1$\steven{,} the assumption that $N$ is saturated implies that
$N$ is simple, and thus $|X| \geq 4$ because $N$ is 4-outwards. \steven{By}
Corollary~\ref{cor:defineSimpleTriplets}, $N$ is $\cL_1(X)$-defined
by $\mathcal R(N)$. 

So assume that $g\geq 2$ and that every saturated, 4-outwards 
network $N\in \cL_1(Y)$ with $g-1$ galls 
is $\cL_1(Y)$-defined by a triplet system on $Y$, where $Y$ is a finite set
of size at least two. Let $N'\in\cL_1(X)$ denote a network for
which $\mathcal R(N) \subseteq \mathcal R(N')$ holds. We need to show
that $N$ and $N'$ are equivalent. To see this, we first analyze the
networks $Collapse(N)$ and $Collapse(N')$.

Note first that,
by Proposition~\ref{prop:samecut},
$Cut(N') = Cut(N)$. Hence, we may assume
without loss of generality that  $X^*$ is 
the leaf set of both $Collapse(N)$ and $Collapse(N')$. 
Next note that
$Collapse(N)$ is 4-outwards
because $N$ is 4-outwards and $|Cut(N)|\geq 4$. Since a simple level-1
network is in particular saturated and  $Collapse(N)$ has precisely
one gall, the base case of the induction implies that $Collapse(N)$
is $\cL(X^*)$-defined by $\mathcal R_1:=\mathcal R(Collapse(N))$. Since
with
 $\mathcal R_2:=\mathcal R(Collapse(N'))$ we have 
$\mathcal R_{Collapse(N)}=\mathcal R_1
\subseteq \mathcal R_2=
\mathcal R_{Collapse(N')}$ and so 
$\mathcal R_1\subseteq \mathcal R_2$ holds it follows that
 $Collapse(N)$
and $Collapse (N')$ must be equivalent.

We next analyze the level-1 networks $N_v$ of $N$ with $v\in X^*$. 
Let $v\in X^*$ and let $C\in Cut(N)$ be such that $v\in C$. Note that
if $|C|=1$ then $N_v$ is an isolated vertex and thus a rooted DAG
with leaf set $\{v\}$. So assume that $|C|\geq 2$. Then since $N$ is
a saturated, 4-outwards network in $\cL_1(X)$, $N_v$ is a saturated, 4-outwards
network in $\cL(C)$. Since $N_v$ has at most $g-1$ galls the induction
hypothesis implies that  $N_v$ is $\cL(C)$-defined
by $\mathcal R(N_v)$. By assumption, $\mathcal R(N) \subseteq \mathcal R(N')$
and so $\mathcal R(N_v) \subseteq \mathcal R(N'_v)$.
Thus $N'_v$ and $N_v$ must be equivalent.  
Combined with the observation
that the networks $Collapse(N)$
and $Collapse (N')$
are equivalent it follows that $N$ and $N'$ are equivalent.
\end{proof}

\section{ $\cL_1(X)$-defining cluster systems}
\label{sec:define-clusters}

In this section, we turn our attention to the companion question
of Section~\ref{sec:triplets-define}. That is, whether
some (not necessarily proper) subset of $\mathcal S(N)$ 
is sufficient to ``uniquely
determine'' a 4-outwards network $N$ in $\cL_1(X)$.
We first present a formalization of the 
idea of ``uniquely determining''  to being $\cL_1(X)$-defined 
for cluster systems. Subsequent to this, we then 
show that all 4-outwards networks $N\in\cL_1(X)$
that are also simple are $\cL_1(X)$-defined
by a cluster system of size at most $|X|$
(Theorem~\ref{theo:simple-clusters} and Corollary~\ref{cor:defineSimple}).
Replacing the requirement that $N$ is simple by the more
general requirement that $N$ is saturated\steven{,} we also show that such networks
are $\cL_1(X)$-defined by their induced softwired cluster system
(Theorem~\ref{theo:shapely-saturated-cluster}).

Let $N$ denote a phylogenetic network on $X$ and let $\mathcal S$
denote a cluster system on $X$.
Then we say that $N$ {\em displays $\mathcal S$ (in
the softwired sense)} if $\mathcal S\subseteq \mathcal S(N)$ holds.
Furthermore,  we say that a network $N\in \cL_1(X)$ is
{\em $\cL_1(X)$-defined by a cluster system $\mathcal S$ on $X$} 
if, up to equivalence, $N$
is the unique network in $\cL_1(X)$ that displays 
$\mathcal S$.
It should be noted that, as in the case of triplet systems,    
a binary phylogenetic tree $T$  on $X$ is not $\cL_1(X)$-defined by 
its induced cluster system $\mathcal C(T)=\mathcal S(T)$. The reason is again
that, by subdividing arcs of $T$ and adding new arcs
joining the subdivision vertices, 
we can transform $T$ into a network $N$ in $\cL_1(X)$
for which  $\mathcal C(T)\subseteq \mathcal S(N)$ holds. 
Also it should be noted that 
a network in $\cL_1(X)$ is not $\cL_1(X)$-defined by its
induced hardwired cluster system. Analogous to the triplet result presented in 
Section~\ref{sec:triplets-define}, a 4-outwards
networks $N\in \cL_1(X)$ also need not be $\cL_1(X)$-defined by
 $\mathcal S(N)$. Indeed, consider again the
two 4-outwards networks $N_1$ and $N_2$
on $X=\{x_1,\ldots, x_5\}$ presented in 
Figures~\ref{fig:softwired-clusters}(i) and 
\ref{fig:simple}(i), respectively. Then $N_1$ and $N_2$
are clearly not equivalent but  $\mathcal S(N_1) \subseteq 
\mathcal S(N_2)$.

\begin{theorem}\label{theo:simple-clusters}
Let $X=\{x_1,\ldots, x_n\}$, $n\geq 4$, 
and suppose that $N$ is a simple network in $\cL_1(X)$ such that, when
starting at the hybrid vertex $v_1$ of $N$ and traversing the unique
cycle $C$ of $U(N)$ counter-clockwise, the 
obtained vertex ordering for 
$C$ is $v_1, v_2, \ldots, v_{i-1}, v_i=\rho_N, v_{i+1},\ldots, v_{n+1}, v_1$
and $x_j$ is a child of $v_j$, for all $1\leq j\leq i-1$, and $x_j$ is
a child of $v_{j+1}$, for all $i\leq j \leq n$. \textcolor{black}{Assume without loss of generality that $i-2\geq n-i+1$ i.e. that the right side of the gall contains
at least as many leaves as the left side. Then}
 $N$ is $\cL_1(X)$-defined by the cluster system $\mathcal S_d(X)$ where
\begin{enumerate}
\item[(i)] 
$\mathcal S_d(N):=\bigcup_{2\leq j\leq n-1}\{\{x_1,x_2,\ldots,x_j\}\} 
\cup\{ \{x_2,x_3,\ldots, x_n\}\}$ if $\rho_N=v_{n+1}$ .
\item[(ii)]
$\mathcal S_d(N):=
\bigcup_{3\leq j\leq n-1}\{\{x_2,x_3\ldots, x_j\}\}
\cup \{\{x_1,x_2\},\{x_1,x_n\},\{x_1,x_2,x_3\}\}$ if $\rho_N=v_n$.
\item[(iii)]
\begin{eqnarray*}
\mathcal S_d(N)&:=&\bigcup_{3\leq j\leq i-1}\{\{x_2,x_3\ldots, x_j\}\}\cup
\bigcup_{n-1 \geq j\geq i}\{\{x_n,x_{n-1},\ldots, x_j\}\}\\
&\cup&
\{\{x_1,x_2\},\{x_1,x_n\},\{x_1,x_n,x_{n-1}\}\}
\end{eqnarray*}
if $\rho_N\not\in\{v_n,v_{n+1}\}$. 
\end{enumerate}
\end{theorem}
\begin{proof}
Let $N_1 \in \mathcal{L}_1(X)$ be a network such that 
$\mathcal{S}_d(N) \subseteq \mathcal{S}(N_1)$. We
first claim that $N_1$ must be simple. Assume for the sake of 
contradiction that $N_1$ is not simple, that is, $N_1$ contains 
a non-trivial cut arc $(u,v)$. Then every cluster in $\mathcal{S}(N_1)$ 
must be compatible with $C= C_{N_1}(v)$, $2 \leq |C| < n$, and 
 $C\in \mathcal{S}(N_1)$. We will derive a contradiction by showing 
that $\mathcal{S}_d(N)$, and thus also $\mathcal{S}(N_1)$, contains 
at least one cluster that is incompatible with
$C$. 

Case (i). We distinguish between the 
two alternatives that $x_1 \in C$ and that $x_1 \not\in C$. 
Assume first that $x_1 \in C$. Then since 
$2 \leq |C| < n$ we have for $C':=\{x_2, \ldots, x_n\}\in \mathcal{S}_d(N)$ 
that $C' \cap C\not=\emptyset$ and that $C' \cap (X\backslash C) \not=\emptyset$,
that is, $C'\not\subseteq C$. Since $x_1\in C$ it follows that 
$C\subseteq C'$ cannot hold either and so $C$ and $C'$ are 
incompatible, as required. Now,
suppose $x_1 \not \in C$. Then since $2 \leq |C|$ there exist 
$p,q\in\{2,\ldots, n\}$ with $p<q$, say, such that 
$x_p, x_q \in C$. Clearly, 
$x_q\not\in C':= \{x_1, \ldots, x_p\}\in \mathcal{S}_d(N)$. But then 
$C'$ and $C$ are again incompatible, as required.

A similar analysis holds for
cases (ii) and (iii); we leave the details to the interested reader. 
Hence, $N_1$ must be simple, as claimed.

Let $h$ denote the unique hybrid vertex of $N_1$ and let $x$ 
denote the leaf of $N_1$ that
is incident with $h$. For the remainder of the proof, we 
consider each of the three cases
stated in the theorem separately. 
All three cases use the following observations: 
(a) if $N_1 |_{X-x}$ is a tree, then all clusters
in $\mathcal{S}( N_1 |_{X-x} )$ are pairwise compatible; 
(b) If $\mathcal{S}_d(N)\subseteq \mathcal{S}(N_1)$, then 
$\mathcal{S}_d(N)|_{X-x}\subseteq \mathcal{S}(N_1)|_{X-x}$ where 
for any cluster system $\mathcal C$ of $X$ we \steven{let}
$\mathcal C|_{X-x} = \{ C \setminus \{x\}\, :\,C \in \mathcal C\}$; 
(c) $\mathcal{S}(N_1 |_{X-x}) = \mathcal{S}(N_1)|_{X-x}$.
For ease of presentation we will liberally make use of the
assumption that $\mathcal{S}_d(N) \subseteq \mathcal{S}(N_1)$ without 
explicitly stating it.

\emph{Case (i)}. First, we argue that $x \in \{x_1, x_2\}$.
Assume for the sake of contradiction that $x \not\in \{x_1, x_2\}$.
Then $C = \{x_1, x_2\}$ and $C' = \{x_2, \ldots, x_n\} \setminus \{x\}$
are incompatible and clearly contained in 
$\mathcal{S}_d(N)|_{X-x}\subseteq \mathcal{S}(N_1)|_{X-x}$. Hence,
$\mathcal{S}(N_1)|_{X-x}$ is not compatible which is impossible
because $x$ is incident with $h$ and so  $N_1 |_{X-x}$ is a 
phylogenetic tree. So $x \in \{x_1, x_2\}$. In fact,  similar
reasoning implies that $x = x_2$
is also impossible as otherwise $\mathcal{S}_d(N)|_{X-x}$ would contain
incompatible clusters $\{x_1, x_3\}$ and $\{x_3, \ldots, x_n\}$. So $x=x_1$. 
Since $\{x_1, x_2\} \subseteq \mathcal{S}_d(N)$ it follows
that the other child of the parent of $x_2$ in $N_1$ is $h$.
Combined with the fact that
$\bigcup_{2 \leq j \leq n-1} \{ \{x_1, x_2, \ldots, x_j\}\} \subseteq \mathcal{S}_d(N)$ it
follows that the other child of the parent of $x_k$ in $N_1$ is the parent of $x_{k-1}$,
$3 \leq k \leq n-1$. Since $\{x_2, \ldots, x_n\} \in \mathcal{S}_d(N)$ it follows
that the other child of the parent of $x_n$ in $N_1$ is the parent of $x_{n-1}$. Hence
$N_1$ is equivalent to $N$.

\emph{Case (ii)}. We claim that $x \in \{x_1, x_2, x_n\}$. The
argument is similar to case (i) in that if $x \not \in \{x_1, x_2, x_n\}$
then $\mathcal{S}_d(N)|_{X-x}\subseteq \mathcal{S}(N_1)|_{X-x}$ contains incompatible
clusters $\{x_1, x_2\}$ and $\{x_1, x_n\}$, leading to a contradiction of 
the fact that $N_1 |_{X-x}$ is a phylogenetic tree.
In fact, similar arguments utilizing the facts that 
$\{x_1, x_2, x_3 \} \in \mathcal{S}_d(N)$  and that $n \neq 3$ imply that
$x \neq x_2$ and $x \neq x_n$. So again $x = x_1$.
 Since $\{x_1, x_2\}$ and $\{x_1, x_n\}$ are contained
in  $\mathcal S_d(N)\subseteq \mathcal S(N_1)$ it follows
that the other child of the parents of $x_2$ and $x_n$ in $N_1$, respectively,
is $h$. In view of  $\{x_1,x_2,x_3\} \subseteq\mathcal S_d(N)
\subseteq \mathcal S(N_1)$ we see that the other child of the parent of $x_3$ in $N_1$ must be the parent of 
$x_2$. Since $\bigcup_{3\leq j\leq n-1}\{\{x_2,x_3\ldots, x_j\}\}
 \subseteq\mathcal S_d(N) $
similar arguments
as in the previous case imply that $N$ and $N_1$ must be equivalent.

\emph{Case (iii)} Again the fact that  $N_1 |_{X-x}$ is a phylogenetic tree
implies  that
$x \in \{x_1, x_2, x_n\}$. However, $x = x_n$ cannot hold 
because $n-1 \neq 2$ and so $\{x_1, x_2\}$ and $\{x_1, x_{n-1}\}$ are distinct clusters
that are both
contained in $\mathcal{S}_d(N)|_{X-x}$ and thus in $ \mathcal{S}(N_1)|_{X-x}$.
But then $ \mathcal{S}(N_1)|_{X-x}$ is incompatible 
which is impossible as   $N_1 |_{X-x}$ is a phylogenetic tree. Similarly,
$x \not= x_2$ as otherwise the two
incompatible clusters $\{x_1, x_n\}$ and $\{x_n, x_{n-1}, \ldots, x_i\}$ 
are contained in $\mathcal{S}_d(N)|_{X-x}$. So $x = x_1$.  Focussing  as in case (ii) on
$x_2$ and $x_n$ we see again that the common child of their respective parents is $h$. Since $\bigcup_{3\leq j\leq i-1}\{\{x_2,x_3\ldots, x_j\}\}\cup
\bigcup_{n-1 \geq j\geq i}\{\{x_n,x_{n-1},\ldots, x_j\}\} 
\subseteq\mathcal S_d(N) $
the location of
the remaining leaves of $N_1$ is forced. Hence, $N_1$ is equivalent to $N$.
%
%
\end{proof}

As an immediate consequence of Theorem~\ref{theo:simple-clusters}, we
obtain the companion result for Theorem~\ref{theo:defineSimpleTriplets}.

\begin{corollary}\label{cor:defineSimple}
Every simple network in $\cL_1(X)$ with at least four 
leaves is $\cL_1(X)$-defined by a cluster system of size at most $|X|$.
\end{corollary}

We now prove the cluster equivalent of 
Theorem~\ref{theo:shapely-saturated-triplet} i.\,e.\,that requiring 
that a 4-outwards network in $\cL_1(X)$ is also saturated guarantees that
it is uniquely determined by its induced softwired cluster system.
%

\begin{theorem}\label{theo:shapely-saturated-cluster}
Every 4-outwards network in $\cL_1(X)$ that is also
saturated is $\cL_1(X)$-defined by its induced softwired cluster system.
\end{theorem}
\begin{proof}
Let $N$ and $N'$ be networks in $\cL_1(X)$ such that $N$ is 
4-outwards and saturated and $\mathcal S(N) \subseteq \mathcal S(N')$ 
holds. We need to show that
 $N'$ is equivalent with $N$. \steven{Let} $\mathcal T = \mathcal T(N)$.
Clearly,
$\bigcup_{T \in \mathcal T} \mathcal S(T)
= \mathcal S(N)\subseteq \mathcal S(N')$. 
Combined with \cite[Proposition 1]{vIK11b} which
implies that $\mathcal R(\mathcal T) \subseteq \mathcal R(N')$
and the fact that $ \mathcal R(N) = \mathcal R(\mathcal{T})$ it follows that
$ \mathcal R(N) \subseteq \mathcal R(N')$.
Since, by Theorem \ref{theo:shapely-saturated-triplet},
$N$ is $\cL_1(X)$-defined by $\mathcal R(N)$ it follows
that $N$ and $N'$ are equivalent.
\end{proof}

In fact, due to the very general character of 
\cite[Proposition 1]{vIK11b}, Theorem~\ref{theo:shapely-saturated-cluster}
can easily be extended to prove that, whenever $\mathcal R(N)$ 
has been proven sufficient to uniquely determine 
(in our sense) a specified subfamily - any 
subfamily - of phylogenetic networks $N$, so too is $\mathcal S(N)$
where we canonically extend the notions of an induced triplet system
and softwired cluster system to such networks.

\section{Conclusions}

In this paper, we have presented enumerative results concerning
the number of vertices, arcs, and galls of a binary level-1 network.
By focusing on triplet systems and (softwired) cluster systems 
we have also investigated the question if subsets 
of those systems suffice to uniquely determine the binary
level-1 network that induced them. As part of this, we have 
presented examples that illustrate that a level-1 network
need not be uniquely determined by the triplet/cluster system
it induces\steven{,} thus illustrating the difference between the notion
of encoding and our formalization of uniquely determining.
In addition, we have provided bounds on the size of such a system 
in case the network in question is simple
and has at least four leaves. For the more general class
of $4$-outwards, saturated, binary level-1 networks
we have shown that any network in that class is uniquely
determined by the triplet/softwired cluster system 
it induces. However, a number of open questions remain. For example
for which binary level-1 networks are the aforementioned bounds sharp
and are 4-outwards saturated binary level-1 networks characterizable
by the fact that they are uniquely determined by their induced 
triplet/softwired cluster system?

We conclude with remarking that in \cite{HM13} {\em trinets}, that 
is, rooted directed acyclic graphs on just three leaves have recently
been introduced in the literature for phylogenetic network 
reconstruction. In that paper it was also shown that any level-1
network is encoded by the trinet system that it induces. In addition, it
was shown in \cite{vIM12} that the more general tree-sibling and
 level-2 networks are encoded by
their induced trinet systems, a fact that is not shared in general for the
triplet system or the softwired cluster system induced by such networks. 
Formalizing the idea of ``uniquely
determining'' for trinet systems in a canonical way to {\em $\cL_1(X)$-defining
trinet systems} it might be
interesting to explore what kind of trinet systems $\cL_1(X)$-define
such networks.

\section*{Acknowledgement}
The authors thank the referee for a very careful reading of the paper. 
KTH and PG thank the London Mathematical Society (LMS) for its support in the
context of its Computer Science Small Grant Scheme.

\bibliographystyle{plain}
\bibliography{GambetteHuberKelk_revision_4april2016}

\begin{thebibliography}{10}

\bibitem{B94}
H.-J. Bandelt.
\newblock Phylogenetic networks.
\newblock {\em Verhandl. Naturwiss. Vereins Hamburg (NF)}, 34:51--71, 1994.

\bibitem{BSS04}
M.~Baroni, C~Semple, and M~Steel.
\newblock A framework for representing reticulate evolution.
\newblock {\em Annals of Combionatorics}, 8:391--408, 2004.

\bibitem{BSS06}
M.~Baroni, C~Semple, and M~Steel.
\newblock Hybrids in real time.
\newblock {\em Systematic Biology}, 55:46--56, 2006.

\bibitem{BM03}
D.~Bryant and V.~Moulton.
\newblock Neighbor-net: An agglomerative method for the construction of
  phylogenetic networrks.
\newblock {\em Molecular Biology and Evolution}, 21(2):255--265, 2003.

\bibitem{CLRV08}
C.~Cardona, M.~Llabres, F.~Rosello, and G.~Valiente.
\newblock A distance metric for a class of tree-sibling phylogenetic networks.
\newblock {\em Bioinformatics}, 24:1481--1488, 2008.

\bibitem{CJSS05}
C.~Choy, J.~Jansson, K.~Sadakane, and W.-K. Sung.
\newblock Computing the maximum agreement of phylogenetic networks.
\newblock {\em Theoretical Computer Science}, 335:93--107, 2005.

\bibitem{DHKMS12}
A.~Dress, K.~T. Huber, V.~Moulton, J.~Koolen, and A.~Spillner.
\newblock {\em Basic Phylogenetic Combinatorics}.
\newblock Cambridge University Press, 2012.

\bibitem{GH12}
P.~Gambette and K.~T. Huber.
\newblock On encodings of phylogenetic networks of bounded level.
\newblock {\em Journal of Mathematical Biology}, 61(1):157--180, 2012.

\bibitem{gusfield2004optimal}
D.~Gusfield, S.~Eddhu, and C.~Langley.
\newblock Optimal, efficient reconstruction of phylogenetic networks with
  constrained recombination.
\newblock {\em Journal of Bioinformatics and Computational Biology},
  2(01):173--213, 2004.

\bibitem{H90}
J.~Hein.
\newblock Reconstructing evolution of sequences subject to recombination using
  parsimony.
\newblock {\em Mathematical Biosciences}, 98:185--200, 1990.

\bibitem{HM13}
K.~T. Huber and V.~Moulton.
\newblock Encoding and constructing 1-nested phylogenetic networks with
  trinets.
\newblock {\em Algorithmica}, 66(3):714--738, 2013.

\bibitem{HIKS2011}
K.~T. Huber, L.~J.~J. van Iersel, S.M. Kelk, and R.~Suchecki.
\newblock A practical algorithm for reconstructing level-1 phylogenetic
  networks.
\newblock {\em IEEE/ACM Transactions in Computational Biology and
  Bioinformatics}, 8(3):607--620, 2011.

\bibitem{HRS10}
D.~Huson, R.~Rupp, and C.~Scornavacca.
\newblock {\em Phylogenetic Networks}.
\newblock Cambridge University Press, 2010.

\bibitem{huson2011survey}
D.~Huson and C.~Scornavacca.
\newblock A survey of combinatorial methods for phylogenetic networks.
\newblock {\em Genome biology and evolution}, 3:23--35, 2011.

\bibitem{JNS06}
J.~Jansson, N.~B. Nguyen, and W.-K. Sung.
\newblock Algorithms for combining rooted triplets into a galled phylogenetic
  network.
\newblock {\em SIAM Journal on Computing}, 35(5):1098--1121, 2006.

\bibitem{M09}
D.~Morrison.
\newblock Phylogenetic networks in systematic biology (and elsewhere).
\newblock {\em Res.\,Adv.\,in Systematic Biology}, 1:1--48, 2009.

\bibitem{rossello2009all}
F.~Rossell{\'o} and G.~Valiente.
\newblock All that glisters is not galled.
\newblock {\em Mathematical biosciences}, 221(1):54--59, 2009.

\bibitem{SS03}
C.~Semple and M.~Steel.
\newblock {\em Phylogenetics}.
\newblock Oxford University Press, 2003.

\bibitem{S75}
P.~Sneath.
\newblock Cladistic representation of reticulate evolution.
\newblock {\em Systematic Zoology}, 24(3):360--368, 1975.

\bibitem{Steel1992}
M.~Steel.
\newblock The complexity of reconstructing trees from qualitative characters
  and subtrees.
\newblock {\em Journal of Classification}, 9:91--116, 1992.

\bibitem{I09}
L.~J.~J. van Iersel.
\newblock {\em Algorithms, Haplotypes and Phylogenetic Networks}.
\newblock PhD thesis, Eindhoven University of Technology, Netherlands, 2009.

\bibitem{IKKSHB09}
L.~J.~J. van Iersel, J.~Keijsper, S.~M. Kelk, L.~Stougie, F.~Hagen, and
  T.~Boekhout.
\newblock Constructing level-2 phylogenetic networks from triplets.
\newblock {\em IEEE/ACM Transactions on Computational Biology and
  Bioinformatics}, 6:1667--681, 2009.

\bibitem{vIK11a}
L.~J.~J. van Iersel and S.~M. Kelk.
\newblock Constructing the simplest possible phylogenetic network from
  triplets.
\newblock {\em Algorithmica}, 60(2):207--235, 2011.

\bibitem{vIK11b}
L.~J.~J. van Iersel and S.~M. Kelk.
\newblock When two trees go to war.
\newblock {\em Journal of Theoretical Biology}, 269(1):245 -- 255, 2011.

\bibitem{vIM12}
L.~J.~J. van Iersel and V.~Moulton.
\newblock Trinets encode tree-child and level-2 phylogenetic networks.
\newblock {\em Journal of Mathematical Biology}, 68(7):1707--1729, 2014.

\bibitem{ISS10}
L.~J.~J. van Iersel, C.~Semple, and M.~Steel.
\newblock Locating a tree in a phylogenetic network.
\newblock {\em Information Processing Letters}, 110(23):1037--1043, 2010.

\bibitem{WZZ01}
L.~Wang, K.~Zhang, and L.~Zhang.
\newblock Perfect phylogenetic networks with recombination.
\newblock {\em Journal of Computational Biology}, 8:69--78, 2001.

\bibitem{W10}
S.~Wilson.
\newblock Properties of normal phylogenetic networks.
\newblock {\em Bulletin of Mathematical Biology}, 72:340--358, 2010.

\end{thebibliography}

\end{document}